\documentclass[a4paper,10pt,english,french]{smfart}
\usepackage[frenchb,english]{babel}
\usepackage[T1]{fontenc}
\usepackage{lmodern}
\usepackage{graphicx}
\usepackage{amsmath,amsthm,amssymb,amsfonts}
\usepackage[a4paper,left=4cm,right=4cm,top=4cm,bottom=4cm]{geometry}
\numberwithin{equation}{section}
\title[Manin's conjecture for a singular cubic surface]{Manin's conjecture for a cubic surface \\ with
$2 \mathbf{A}_2 + \mathbf{A}_1$ singularity type}
\author{Pierre Le Boudec}
\subjclass{$11$D$45$, $14$G$05$}
\address{Université Denis Diderot (Paris VII) \\ Institut de Mathématiques de Jussieu \\ UMR 7586 \\ Case $7012$ - Bâtiment Chevaleret \\ Bureau $7$C$14$ \\ $75205$ Paris Cedex 13}
\email{pleboude@math.jussieu.fr}

\begin{document}

\makeatletter
\def\imod#1{\allowbreak\mkern10mu({\operator@font mod}\,\,#1)}
\makeatother

\newtheorem{lemma}{Lemma}
\newtheorem{theorem}{Theorem}
\newtheorem{corollaire}{Corollaire}
\newtheorem{proposition}{Proposition}

\newcommand{\vol}{\operatorname{vol}}
\newcommand{\D}{\mathrm{d}}
\newcommand{\rank}{\operatorname{rank}}
\newcommand{\Pic}{\operatorname{Pic}}
\newcommand{\Gal}{\operatorname{Gal}}
\newcommand{\meas}{\operatorname{meas}}
\newcommand{\Spec}{\operatorname{Spec}}
\newcommand{\eff}{\operatorname{eff}}

\begin{abstract}
We establish Manin's conjecture for a cubic surface split over $\mathbb{Q}$ and whose singularity type is $2 \mathbf{A}_2 + \mathbf{A}_1$. For this, we make use of a deep result about the equidistribution of the values of a certain restricted divisor function in three variables in arithmetic progressions. This result is due to Friedlander and Iwaniec \cite{MR786351} (and was later improved by Heath-Brown \cite{MR866901}) and draws on the work of Deligne \cite{MR0340258}.
\end{abstract}

\maketitle

\tableofcontents

\section{Introduction}

At the end of the eighties, Manin and his collaborators initiated a program aiming to investigate the distribution of rational points on Fano varieties (see \cite{MR974910}) and they gave a precise conjecture concerning the asymptotic behaviour of the number of rational points of bounded height. In this paper we focus on the case of singular del Pezzo surfaces of degree three defined over $\mathbb{Q}$. In order to precisely state the conjecture in this case, we introduce the exponential height
$H : \mathbb{P}^3(\mathbb{Q}) \to \mathbb{R}_{> 0}$ which is defined for a vector $(x_0, x_1, x_2, x_3) \in \mathbb{Z}^{4}$ subject to the condition
$\gcd(x_0, x_1, x_2, x_3) = 1$ by
\begin{eqnarray*}
H(x_0: x_1 : x_2 : x_3) & = & \max \{ |x_0|, |x_1|, |x_2|, |x_3| \}\textrm{.}
\end{eqnarray*}
Let $V \subset \mathbb{P}^3$ be a singular cubic surface defined over $\mathbb{Q}$. The variety $V$ contains a positive number of lines where the rational points accumulate hiding the distribution of the rational points on the complement of the lines. To surpass this phenomenon, we let $U$ be the open subset formed by removing the lines from $V$ and we define the quantity
\begin{eqnarray*}
N_{U,H}(B) & = & \# \{x \in U(\mathbb{Q}), H(x) \leq B \} \textrm{.}
\end{eqnarray*}
It is the number of rational points on $V$ which do not lie on any line and whose height is bounded by a quantity $B$ which has to be thought as tending to infinity. If $\widetilde{V}$ denotes the minimal desingularization of $V$ and $\rho = \rho_{\widetilde{V}}$ the rank of its Picard group, then it is expected that
\begin{eqnarray*}
N_{U,H}(B) & = & c_{V,H} B \log(B)^{\rho - 1} (1+o(1)) \textrm{,}
\end{eqnarray*}
where $c_{V,H}$ is a constant whose value is expected to follow Peyre's prediction \cite{MR1340296}. In comparison, it is easy to see that the number $N_{\mathbb{P}^1,H}(B)$ of rational points of bounded height lying on a line satisfies
$N_{\mathbb{P}^1,H}(B) = c_{\mathbb{P}^1,H} B^2 (1+o(1))$ where $c_{\mathbb{P}^1,H} > 0$. That is why we needed to exclude the rational points lying on the lines of $V$ from the counting function.

The classification of singular cubic surfaces is classical and goes back to Schläfli \cite{Schläfli} and Cayley \cite{Cayley}, they are simply categorized by their singularity types. This classification is described in a modern language in the work of Bruce and Wall \cite{MR533323}. Up to isomorphism over $\overline{\mathbb{Q}}$, there are twenty different singularity types (see
\cite[Table $5$]{D-hyper} for instance). Note that for some types, there can be several isomorphism classes of surfaces (and even infinite families). From now on, we restrict our attention to surfaces which are split over $\mathbb{Q}$ meaning that their singularities and the lines they contain are defined over $\mathbb{Q}$. Despite the growing interest borne to Manin's conjecture for del Pezzo surfaces, the conjecture has only been proved for three cubic surfaces of different types. This is due to the fact that there is no general method to check that a given variety satisfies the conjecture. However, there exist some general results asserting that the conjecture holds for certain large classes of varieties. For instance, Batyrev and Tschinkel have proved it for toric varieties \cite{MR1620682} and Chambert-Loir and Tschinkel for equivariant compactifications of vector groups \cite{MR1906155}. With this end in view, they study the height Zeta function of the variety
\begin{eqnarray*}
Z_{U,H}(s) & = & \sum_{x \in U(\mathbb{Q})} H(x)^{-s} \textrm{,}
\end{eqnarray*}
using harmonic analysis techniques in an adelic setting. It turns out that the surface having singularity type $3 \mathbf{A}_2$ and whose equation is
\begin{eqnarray*}
x_0^3 & = & x_1 x_2 x_3 \textrm{,}
\end{eqnarray*}
is toric and thus the result of Batyrev and Tschinkel covers this case. However, many authors have studied the quantity $N_{U,H}(B)$ for this particular surface (see \cite{MR1679838}, \cite{MR1679841}, \cite{MR1679839}, \cite{MR1656797} and \cite{MR2430199}) and have obtained stronger results. The best of these results is due to la Bretèche \cite{MR1679839} who has proved that the height Zeta function of this surface admits a meromorphic continuation on the left of the line $\Re(s) = 1$ and moreover that there exists a monic polynomial $P$ of degree $6 = \rho - 1$ and a constant $c > 0$ such that
\begin{eqnarray*}
N_{U,H}(B) & = & c_{V,H} B P(\log(B)) + O \left( B^{7/8} \exp \left( -c \mathcal{L}(B) \right) \right) \textrm{,}
\end{eqnarray*}
where $\mathcal{L}(B) = \log(B)^{3/5} \log(\log(B))^{-1/5}$. Manin's conjecture has also been proved for a cubic surface with $\mathbf{E}_6$ singularity type and whose equation is
\begin{eqnarray*}
x_1 x_2^2 + x_2 x_0^2 + x_3^3 & = & 0 \textrm{.}
\end{eqnarray*}
It was first proved by Derenthal in his doctoral thesis \cite{Der-th} (and independently by Joyce, also in his doctoral thesis
\cite{Joyce-th}) and then la Bretèche, Browning and Derenthal \cite{MR2332351} obtained a much stronger result. They proved that the height Zeta function of this surface can also be meromorphically continued on the left of $\Re(s) = 1$ and that there exists a monic polynomial $Q$ of degree $6$ such that for any fixed $\varepsilon > 0$,
\begin{eqnarray}
\label{polynomial}
N_{U,H}(B) & = & c_{V,H} B Q(\log(B)) + O \left( B^{10/11 + \varepsilon} \right) \textrm{.}
\end{eqnarray}
Finally, Browning and Derenthal \cite{MR2520769} have obtained Manin's conjecture for a surface having singularity type $\mathbf{D}_5$ and whose equation is
\begin{eqnarray*}
x_3 x_0^2 + x_0 x_2^2 + x_1^2 x_2 & = & 0 \textrm{.}
\end{eqnarray*}

It is a general expectation that it is easier to get a good understanding of the asymptotic behaviour of $N_{U,H}(B)$ when the surface has a \textit{strong singularity type}. According to this principle, it seems hard to reach a result similar to \eqref{polynomial} for the latter surface or any other having a \textit{weaker singularity type}. To support this heuristic fact, we can note that the only results which are available for certain \textit{less singular} cubic surfaces are lower and upper bounds of the expected order of magnitude. To be more precise, for Cayley's cubic surface which has singularity type $4 \mathbf{A}_1$ and whose equation is
\begin{eqnarray*}
x_0 x_1 x_2 + x_0 x_1 x_3 + x_0 x_2 x_3 + x_1 x_2 x_3 & = & 0 \textrm{,}
\end{eqnarray*}
Heath-Brown \cite{MR2075628} has proved that
\begin{eqnarray*}
N_{U,H}(B) & \asymp & B \log(B)^6 \textrm{,}
\end{eqnarray*}
meaning that the ratio of these two quantities is between two constants. In addition, Browning \cite{MR2250046} has obtained exactly the same result for a surface having a $\mathbf{D}_4$ singularity and which is given by
\begin{eqnarray*}
x_0 (x_1 + x_2 + x_3)^2 - x_1 x_2 x_3 & = & 0 \textrm{.}
\end{eqnarray*}

The proofs of all these results are intrinsically very different from the proof of Batyrev and Tschinkel for toric varieties. They all use a passage to universal torsors. This consists in defining a bijection between the set of rational points to be counted on $U$ and a certain set of integral points of an affine variety of higher dimension, which is equal to nine for cubic surfaces. In the cases for which the universal torsors are hypersurfaces, Derenthal has calculated their equations for all singular cubic surfaces (see
\cite{D-hyper}). This task is achieved determining the total coordinate ring associated to the minimal desingularizations of the surfaces using a method of Hassett and Tschinkel \cite{MR2029868}. However, this step of the proof can also be carried out using elementary techniques (see section \ref{torsor section} for an example).

The aim of this paper is to prove Manin's conjecture for another cubic surface split over $\mathbb{Q}$ and having singularity type $2 \mathbf{A}_2 + \mathbf{A}_1$. This surface $V \subset \mathbb{P}^3$ contains five lines and is defined by
\begin{eqnarray*}
x_3^2 (x_1 + x_3) + x_0 x_1 x_2 & = & 0 \textrm{.}
\end{eqnarray*}
The lines on $V$ are given by $x_i = x_3 = 0$ and $x_j = x_1 + x_3 = 0$ for $i \in \{0, 1, 2 \}$ and $j \in \{0, 2 \}$. Its three singularities are $(0:1:0:0)$, $(1:0:0:0)$ and $(0:0:1:0)$. It is easy to see that the first has type $\mathbf{A}_1$ and the two others have type $\mathbf{A}_2$. We also see that $V$ is actually split over $\mathbb{Q}$ and thus, if $\widetilde{V}$ denotes the minimal desingularization of $V$, the Picard group of $\widetilde{V}$ has rank $\rho = 7$. The open subset $U$ and the quantity $N_{U,H}(B)$ we want to investigate are defined as explained above. As already said, in section \ref{torsor section}, we define a bijection between the set of the points we aim to count on $U$ and a certain set of integral points of an open subset of the affine hypersurface embedded in $\mathbb{A}^{10} \simeq \Spec \left( \mathbb{Q}[\eta_1, \dots, \eta_{10}] \right)$ and defined by
\begin{eqnarray*}
\eta_1 \eta_6 \eta_8 + \eta_3 \eta_5 \eta_7^2 + \eta_9 \eta_{10} & = & 0 \textrm{.}
\end{eqnarray*}

The first step of the proofs of Manin's conjecture for the $\mathbf{E}_6$ and the $\mathbf{D}_5$ cubic surfaces mentioned above consists in summing over two variables seeing the torsor equation as a congruence and counting the number of integers in a prescribed region and subject to this congruence. It seems highly unlikely that this method turns out to be efficient in our case and to overcome this obstacle, we start by summing over four variables at once. Note that starting by summing over more than two variables has already proved to be an efficient strategy to count integral points on universal torsors (for example in \cite{3A1} and far more strikingly in \cite{dP2E7}).

To make our proof work, we need a deep result of Friedlander and Iwaniec \cite[section $3$, proposition $1$]{MR786351} (restated in lemma \ref{Fri-Iwa lemma}) concerning the distribution of the values of a certain restricted divisor function in arithmetic progressions. This function is similar to the divisor function $\tau_3 := \tau \ast 1$ where $\tau$ denotes the usual divisor function, apart from the fact that the divisors counted have to lie in a prescribed region of~$\mathbb{R}^3$. To reach this result, Friedlander and Iwaniec combine the use of the work of Deligne \cite{MR0340258} to deal with some complete exponential sums with some other ideas to analyze incomplete Kloosterman sums, ideas which were first developed by Burgess in the context of sums of multiplicative characters \cite{MR0132732}. It is certainly worth underlining that using a bound for two-dimensional Kloosterman sums consequence of Deligne's work in the most obvious way, as in the previous work of the author \cite[Lemma $1$]{3A1} where Weil's bound for usual Kloosterman sums is used, would not have been sufficient for our purpose. We have decided to use the result of Friedlander and Iwaniec because it was stated as we needed but it is important to note that the work of Heath-Brown \cite{MR866901} might have been used instead. Even slightly better, his result would not have improved our final error term. Our main result is the following.

\begin{theorem}
\label{Manin}
As $B$ tends to $+ \infty$, we have the estimate
\begin{eqnarray*}
N_{U,H}(B) & = & c_{V,H} B \log(B)^{6} \left( 1 + O \left( \frac{\log(\log(B))}{\log(B)} \right) \right) \textrm{,}
\end{eqnarray*}
where $c_{V,H}$ matches Peyre's prediction.
\end{theorem}

Since $\rho = 7$, this estimate proves that $V$ satisfies Manin's conjecture. Derenthal has proved that $V$ is not toric \cite[Proposition 12]{D-hyper} and Derenthal and Loughran have proved that it is not an equivariant compactification of $\mathbb{G}_a^2$ \cite{DL-equi}, so theorem~\ref{Manin} is not a consequence of the general results concerning equivariant compactifications of algebraic groups \cite{MR1620682} and \cite{MR1906155}.

The following section is dedicated to the proofs of several preliminary results. The most important part of it is section \ref{equidistribution section} in which we present the result of Friedlander and Iwaniec on which relies the proof of theorem \ref{Manin}. In the next two sections, we respectively introduce the universal torsor mentioned previously and calculate Peyre's constant. Finally, the remaining section is devoted to the proof of theorem \ref{Manin}.

It is a pleasure for the author to thank his supervisor Professor de la Bretèche for his high availability and his enthusiastic guidance. The author is also grateful to Professor Derenthal for his useful explanations about the value of the constant $\alpha(\widetilde{V})$ appearing in Peyre's constant.

The financial support of the ANR PEPR (Points Entiers Points Rationnels) is gratefully acknowledged.

\section{Preliminaries}

\subsection{An elementary lemma}

We state the following elementary result as it may turn out to be useful for further applications. We will use it in section \ref{summations section} in the case where $r = 3$.

\begin{lemma}
\label{elementary}
Let $A \in \mathbb{R}$, $Y \geq 1$ and $r \in \mathbb{Z}_{\geq 1}$. Let also $\mathcal{R} \subset \mathbb{R}$ be the set of real numbers $t$ subject to the condition
\begin{eqnarray}
\label{cond}
\left| t^r + A t^{r-1} \right| & \leq & Y \textrm{.}
\end{eqnarray}
We have the bound
\begin{eqnarray*}
\meas(\mathcal{R}) & \leq & 4 Y^{1/r} \textrm{.}
\end{eqnarray*}
\end{lemma}

\begin{proof}
Let $I = \left\{ t \in \mathbb{R}, |t| \leq Y^{1/r} \right\}$ and $J = \mathbb{R} \setminus I$. Since $\meas \left( I \right) = 2 Y^{1/r}$, we have $\meas \left( \mathcal{R} \cap I \right) \leq 2 Y^{1/r}$. Moreover, if $t \in J$, the condition \eqref{cond} gives $\left| t + A \right| \leq Y^{1/r}$. This shows that $\meas \left( \mathcal{R} \cap J \right) \leq 2 Y^{1/r}$, which completes the proof.
\end{proof}

\subsection{Equidistribution of the values of a restricted divisor function in \text{arithmetic} progressions}

\label{equidistribution section}

As already underlined, the proof of theorem \ref{Manin} draws upon a deep result about the equidistribution of the values of a certain divisor function in arithmetic progressions. From now on, let $0 < \delta \leq 1$ be a parameter, $\zeta = 1 + \delta$ and let $U$, $V$ and $W$ be variables running over the set $\{ \pm \zeta^n, n \in \mathbb{Z}_{\geq -1} \}$. We define $\mathcal{I} = ]U,\zeta U]$ if $U > 0$ and $\mathcal{I} = [\zeta U,U[$ if $U < 0$. The ranges $\mathcal{J}$ and $\mathcal{K}$ are built the same way using respectively the variables $V$ and $W$. Let also $a, q \in \mathbb{Z}_{\geq 1}$ be two coprime integers. We introduce the two quantities
\begin{eqnarray}
\label{N}
\ \ \ \ \ \ N(\mathcal{I},\mathcal{J},\mathcal{K};q,a) & = &
\# \left\{ (u,v,w) \in \mathcal{I} \times \mathcal{J} \times \mathcal{K} \cap \mathbb{Z}^3,
uvw \equiv a \imod{q} \right\} \textrm{,}
\end{eqnarray}
and
\begin{eqnarray}
\label{Nast}
\ \ \ \ \ \ N^{\ast}(\mathcal{I},\mathcal{J},\mathcal{K};q) & = &
\frac1{\varphi(q)} \# \left\{ (u,v,w) \in \mathcal{I} \times \mathcal{J} \times \mathcal{K} \cap \mathbb{Z}^3,
\gcd(uvw,q) = 1 \right\} \textrm{.}
\end{eqnarray}
The following lemma is a restatement of \cite[section $3$, proposition $1$]{MR786351}.

\begin{lemma}
\label{Fri-Iwa lemma}
Let $U, V, W$ be as described above and let $X > 0$ be a quantity such that $|UVW| \leq X$. Define
\begin{eqnarray}
\label{definition E}
E(X,q) & = & \frac{X^{1/2}}{q^{1/150}} \left( \frac{X}{q^2} \right)^{17/150} \textrm{.}
\end{eqnarray}
Let $\varepsilon > 0$ be fixed. For $q \leq X^{1/2 + 1/230}$, we have the estimate
\begin{eqnarray*}
N(\mathcal{I},\mathcal{J},\mathcal{K};q,a) & = & N^{\ast}(\mathcal{I},\mathcal{J},\mathcal{K};q)
+ O \left( X^{\varepsilon} E(X,q) \right) \textrm{.}
\end{eqnarray*}
\end{lemma}

Let us remark that the proof of this lemma actually shows that the result is also true if $U$, $V$ and $W$ are any non-zero quantities.

It is easy to check that this result is stronger than the result obtained by a more straightforward appeal to Deligne's work only for $q > X^{1/2 - 1/370}$. Note that in \cite{MR786351}, Friedlander and Iwaniec only work with positive $u$, $v$ and $w$ but this does not change anything since we can change $a$ in $-a$. Note also that an immediate consequence of this estimate is
\begin{eqnarray}
\label{bound}
N(\mathcal{I},\mathcal{J},\mathcal{K};q,a) & \ll & \frac1{\varphi(q)} \# \left( \mathcal{I} \times \mathcal{J} \times \mathcal{K}  \cap \mathbb{Z}^3 \right) + X^{\varepsilon} E(X,q) \textrm{.}
\end{eqnarray}

We now introduce a certain domain $\mathcal{S} \subset \mathbb{R}^3$ where the triple $(u,v,w)$ is restricted to lie. Let
$X, X_1, X_2, T, Z, L_1, L_2 > 0$. We let $\mathcal{S} = \mathcal{S}(X, X_1, X_2, T, Z, L_1, L_2)$ be the set of
$(x,y,z) \in \mathbb{R}_{\geq 0}^2 \times \mathbb{R}$ such that
\begin{eqnarray}
\label{A}
x y^2 |xyz + T| & \leq & X_1 \textrm{,} \\
\label{B}
x z^2 & \leq & X_2 \textrm{,} \\
\label{C}
xy|z| & \leq & X \textrm{,} \\
\label{D}
Z & \leq & |xyz + T| \textrm{,} \\
\label{E}
L_1 & \leq & y \textrm{,} \\
\label{F}
L_2 & \leq & |z| \textrm{.}
\end{eqnarray}
Finally, we introduce
\begin{eqnarray*}
D(\mathcal{S};q,a) & = & \# \left\{ (u,v,w) \in \mathcal{S} \cap \mathbb{Z}^3, uvw \equiv a \imod{q} \right\} \textrm{,}
\end{eqnarray*}
and
\begin{eqnarray*}
D^{\ast}(\mathcal{S};q) & = & \frac1{\varphi(q)} \# \left\{ (u,v,w) \in \mathcal{S} \cap \mathbb{Z}^3, \gcd(uvw,q) = 1 \right\} \textrm{.}
\end{eqnarray*}

\begin{lemma}
\label{pre lemma tau}
Let $\varepsilon > 0$ be fixed. If $T \leq X$ then for $q \leq X^{1/2 + 1/230}$, we have the estimate
\begin{eqnarray*}
D(\mathcal{S};q,a) - D^{\ast}(\mathcal{S};q) & \ll & \frac{X^{1/2 + \varepsilon}}{q^{1/600}}
\left( \frac{X}{q^2} \right)^{1/2- \vartheta} + \frac{X \log(X)}{\varphi(q)} \left( \frac1{L_1} + \frac1{L_2} \right) \textrm{,}
\end{eqnarray*}
where $\vartheta = 29/300$.
\end{lemma}

Note that the assumption $T \leq X$ together with the two conditions $xy|z| \leq X$  and $Z \leq |xyz + T|$ imply $Z \leq 2 X$.

\begin{proof}
If $\mathcal{S} \cap \mathbb{Z}_{\neq 0}^3 = \emptyset$ then the result obviously holds, we therefore assume from now on that
$\mathcal{S} \cap \mathbb{Z}_{\neq 0}^3 \neq \emptyset$. We let $0 < \delta \leq 1$ be a parameter to be selected later. Recall the definitions of $\zeta$, $U$, $V$, $W$ and $\mathcal{I}$, $\mathcal{J}$, $\mathcal{K}$ given at the beginning of the section. We have
\begin{eqnarray*}
D(\mathcal{S};q,a) -  \sum_{\mathcal{I} \times \mathcal{J} \times \mathcal{K} \cap \mathbb{Z}^3 \subset \mathcal{S}}
N(\mathcal{I},\mathcal{J}, \mathcal{K};q,a) & \ll &
\sum_{\substack{\mathcal{I} \times \mathcal{J} \times \mathcal{K} \cap \mathbb{Z}^3 \nsubseteq \mathcal{S} \\ \mathcal{I} \times \mathcal{J} \times \mathcal{K} \cap \mathbb{Z}^3 \nsubseteq \mathbb{R}^3 \setminus \mathcal{S}}}
N(\mathcal{I},\mathcal{J},\mathcal{K};q,a) \textrm{.}
\end{eqnarray*}
We define the quantity
\begin{eqnarray*}
D(\mathcal{S};q) & = & \sum_{\mathcal{I} \times \mathcal{J} \times \mathcal{K} \cap \mathbb{Z}^3 \subset \mathcal{S}}
N^{\ast}(\mathcal{I},\mathcal{J},\mathcal{K};q) \textrm{.}
\end{eqnarray*}
We note that since $N^{\ast}(\mathcal{I},\mathcal{J},\mathcal{K};q)$ is independent of $a$, so is $D(\mathcal{S};q)$. Moreover, we have
\begin{eqnarray*}
\sum_{\mathcal{I} \times \mathcal{J} \times \mathcal{K} \cap \mathbb{Z}^3 \subset \mathcal{S}}
N(\mathcal{I},\mathcal{J},\mathcal{K};q,a) - D(\mathcal{S};q) & \ll & \frac{X^{\varepsilon} E(X,q)}{\delta^3} \textrm{,}
\end{eqnarray*}
using lemma \ref{Fri-Iwa lemma} and noticing that the number of hyperrectangles $\mathcal{I} \times \mathcal{J} \times \mathcal{K}$ such that $\mathcal{I} \times \mathcal{J} \times \mathcal{K} \cap \mathbb{Z}^3 \subset \mathcal{S}$ is less than
$2 \left( \log(X) / \log(\zeta) \right)^3 \ll X^{\varepsilon} \delta^{-3}$ since $\delta \leq 1$. We have proved that \begin{eqnarray*}
D(\mathcal{S};q,a) - D(\mathcal{S};q) & \ll &
\sum_{\substack{\mathcal{I} \times \mathcal{J} \times \mathcal{K} \cap \mathbb{Z}^3 \nsubseteq \mathcal{S} \\
\mathcal{I} \times \mathcal{J} \times \mathcal{K} \cap \mathbb{Z}^3 \nsubseteq \mathbb{R}^3 \setminus \mathcal{S}}}
N(\mathcal{I},\mathcal{J},\mathcal{K};q,a) + \frac{X^{\varepsilon} E(X,q)}{\delta^3} \textrm{.}
\end{eqnarray*}
Using the bound \eqref{bound} for $N(\mathcal{I},\mathcal{J},\mathcal{K};q,a)$, we conclude that
\begin{eqnarray*}
D(\mathcal{S};q,a) - D(\mathcal{S};q) & \ll & \frac1{\varphi(q)}
\sum_{\substack{\mathcal{I} \times \mathcal{J} \times \mathcal{K} \cap \mathbb{Z}^3 \nsubseteq \mathcal{S} \\ \mathcal{I} \times \mathcal{J} \times \mathcal{K} \cap \mathbb{Z}^3 \nsubseteq \mathbb{R}^3 \setminus \mathcal{S}}}
\# \left( \mathcal{I} \times \mathcal{J} \times \mathcal{K} \cap \mathbb{Z}^3 \right) + \frac{X^{\varepsilon} E(X,q)}{\delta^3} \textrm{,}
\end{eqnarray*}
since the number of hyperrectangles $\mathcal{I} \times \mathcal{J} \times \mathcal{K}$ satisfying
$\mathcal{I} \times \mathcal{J} \times \mathcal{K} \cap \mathbb{Z}^3 \nsubseteq \mathcal{S}$ and 
$\mathcal{I} \times \mathcal{J} \times \mathcal{K} \cap \mathbb{Z}^3 \nsubseteq \mathbb{R}^3 \setminus \mathcal{S}$ is also
$\ll X^{\varepsilon} \delta^{-3}$. The sum of the right-hand side is over all the hyperrectangles
$\mathcal{I} \times \mathcal{J} \times \mathcal{K}$ for which we have
$(\zeta^{s_1} U, \zeta^{s_2} V, \zeta^{s_3} W) \in \mathcal{S} \cap \mathbb{Z}^3$ and
$(\zeta^{t_1} U, \zeta^{t_2} V, \zeta^{t_3} W) \in \mathbb{Z}^3 \setminus \mathcal{S}$ for some triples $(s_1,s_2,s_3) \in ]0,1]^3$ and $(t_1,t_2,t_3) \in ]0,1]^3$. This means that one of the inequalities defining $\mathcal{S}$ is not satisfied by
$(\zeta^{t_1} U, \zeta^{t_2} V, \zeta^{t_3} W)$ and we need to estimate the contribution coming from each condition among \eqref{A}, \eqref{B}, \eqref{C}, \eqref{D}, \eqref{E} and \eqref{F}. Note that combining the conditions \eqref{C}, \eqref{E} and \eqref{F}, we get
\begin{eqnarray}
\label{condition0'}
U|W| & \ll & \frac{X}{L_1} \textrm{,} \\
\label{condition0}
UV & \ll & \frac{X}{L_2} \textrm{.}
\end{eqnarray}
Note that, in what follows, we could sometimes write strict inequalities instead of non-strict but this would not change anything in our reasoning. Let us first deal with the condition \eqref{A}. For the hyperrectangles $\mathcal{I} \times \mathcal{J} \times \mathcal{K}$ described above, for some $(s_1,s_2,s_3) \in ]0,1]^3$ and $(t_1,t_2,t_3) \in ]0,1]^3$, we have
\begin{eqnarray}
\label{condition1}
\zeta^{s_1 + 2s_2} U V^2 \left| \zeta^{s_1 + s_2 + s_3} UVW + T \right| & \leq & X_1 \textrm{,} \\
\label{condition2}
\zeta^{t_1 + 2t_2} U V^2 \left| \zeta^{t_1 + t_2 + t_3} UVW + T \right| & > & X_1 \textrm{.}
\end{eqnarray}
Note that using $T \leq X$ and $UV|W| \leq X$, the second inequality gives
\begin{eqnarray}
\label{condition3}
U V^2 & \gg & \frac{X_1}{X} \textrm{.}
\end{eqnarray}
The two conditions \eqref{condition1} and \eqref{condition2} imply
\begin{eqnarray}
\label{condition4}
& & \zeta^{-6} \frac{X_1}{U V^2} - \left( 1 - \zeta^{-3} \right) T
< \left| UVW + T \right| \leq \frac{X_1}{U V^2} + \left( 1 - \zeta^{-3} \right) T \textrm{.}
\end{eqnarray}
Going back to the variables $u$, $v$ and $w$, we easily get
\begin{eqnarray*}
\big| \left| uvw + T \right| - \left| UVW + T \right| \big| & \leq & 
\left| uvw - UVW \right| \\
& \leq & 7 \delta UV|W| \\
& \leq & 7 \delta \left( \frac{X_1}{UV^2} + T \right) \textrm{,}
\end{eqnarray*}
using the condition \eqref{condition1}. Since $1 - \zeta^{-3} \leq 3 \delta$, the inequality \eqref{condition4} gives
\begin{eqnarray*}
& & \left( \zeta^{-6} -  7 \delta \right) \frac{X_1}{U V^2} - 10 \delta T < \left| uvw + T \right| \leq
\left( 1 + 7 \delta \right) \frac{X_1}{U V^2} + 10 \delta T \textrm{,}
\end{eqnarray*}
and therefore
\begin{eqnarray}
\label{condition uvw}
& & \left( \zeta^{-6} -  7 \delta \right) \frac{X_1}{u v^2} - 10 \delta T < \left| uvw + T \right| \leq
\zeta^3 \left( 1 + 7 \delta \right) \frac{X_1}{u v^2} + 10 \delta T \textrm{,}
\end{eqnarray}
Note that this inequality is not as sharp as possible but it does not matter for our purpose. Thereby, we see that the error we want to estimate is bounded by
\begin{eqnarray*}
\sum_{\substack{\eqref{condition0} \\ \eqref{condition3}, \eqref{condition4}}}
\# \left( \mathcal{I} \times \mathcal{J} \times \mathcal{K} \cap \mathbb{Z}^3 \right)
& \ll & \# \left\{ (u,v,w) \in \mathbb{Z}_{\neq 0}^3,
\begin{array}{l}
\eqref{condition uvw} \\
uv \ll X / L_2 \\
u v^2 \gg X_1 / X
\end{array}
\right\} \\
& \ll & \sum_{\substack{uv \ll X / L_2 \\ u v^2 \gg X_1 / X}}
\left( \frac{\delta X_1}{u^2v^3} + \frac{\delta T}{uv} + 1 \right) \\
& \ll & \sum_{v \ll X / L_2}
\left( \frac{\delta X}{v} + \frac{\delta T X^{\varepsilon}}{v} + \frac{X}{L_2 v} \right) \\
& \ll & \delta X^{1 + \varepsilon} + \frac{X \log(X)}{L_2} \textrm{,}
\end{eqnarray*}
since $T \leq X$. We now reason in a similar way to deal with the other conditions. Let us estimate the contribution coming from the condition \eqref{B}. We see that the condition which plays the role of \eqref{condition4} in the previous case is here
\begin{eqnarray}
\label{condition5}
& & \zeta^{-3} X_2 < U W^2 \leq X_2 \textrm{,}
\end{eqnarray}
and, combined with $U V |W| \leq X$, it implies
\begin{eqnarray}
\label{condition6}
U^{1/2} V & \ll & \frac{X}{X_2^{1/2}} \textrm{.}
\end{eqnarray}
Furthermore, going back to the variables $u$ and $w$, we have $\zeta^{-3} X_2 < u w^2 \leq \zeta^3 X_2$. We therefore find that the error in this case is bounded by
\begin{eqnarray*}
\sum_{\substack{\eqref{condition0} \\ \eqref{condition5}, \eqref{condition6}}}
\# \left( \mathcal{I} \times \mathcal{J} \times \mathcal{K} \cap \mathbb{Z}^3 \right)
& \ll & \# \left\{ (u,v,w) \in \mathbb{Z}_{\neq 0}^3,
\begin{array}{l}
\zeta^{-3} X_2 < u w^2 \leq \zeta^3 X_2 \\
uv \ll X / L_2 \\
u^{1/2} v \ll X / X_2^{1/2}
\end{array}
\right\} \\
& \ll & \sum_{\substack{uv \ll X / L_2 \\ u^{1/2} v \ll X / X_2^{1/2}}} \left( \frac{\delta X_2^{1/2}}{u^{1/2}} + 1 \right) \\
& \ll &  \sum_{v \ll X / L_2} \left( \frac{\delta X}{v} + \frac{X}{L_2 v} \right) \\
& \ll & \delta X^{1 + \varepsilon} + \frac{X \log(X)}{L_2} \textrm{.}
\end{eqnarray*}
In the case of the condition \eqref{C}, we have the inequality
\begin{eqnarray}
\label{condition7}
& & \zeta^{-3} X < UV|W| \leq X \textrm{,}
\end{eqnarray}
and the condition on the variables $u$, $v$ and $w$ is thus $\zeta^{-3} X < uv|w| \leq \zeta^3 X$. In a similar fashion, we see that the contribution corresponding to this condition is bounded by
\begin{eqnarray*}
\sum_{\eqref{condition0}, \eqref{condition7}}
\# \left( \mathcal{I} \times \mathcal{J} \times \mathcal{K} \cap \mathbb{Z}^3 \right) & \ll &
\# \left\{ (u,v,w) \in \mathbb{Z}_{\neq 0}^3,
\begin{array}{l}
\zeta^{-3} X < uv|w| \leq \zeta^3 X \\
uv \ll X / L_2
\end{array}
\right\} \\
& \ll & 
\sum_{uv \ll X / L_2} \left( \frac{\delta X}{u v} + 1 \right) \\
& \ll & \sum_{v \ll X / L_2} \left( \frac{\delta X^{1 + \varepsilon}}{v} + \frac{X}{L_2 v} \right) \\
& \ll & \delta X^{1 + \varepsilon} + \frac{X \log(X)}{L_2} \textrm{.}
\end{eqnarray*}
Let us now deal with the condition \eqref{D}. This time, reasoning as we did to obtain the condition \eqref{condition4}, we get
\begin{eqnarray}
\label{condition8}
& & \zeta^{-3} Z - \left( 1 - \zeta^{-3} \right) T
\leq \left| UVW + T \right| < Z + \left( 1 - \zeta^{-3} \right) T \textrm{.}
\end{eqnarray}
We can reason exactly as we did to derive the inequality \eqref{condition uvw} from \eqref{condition4}. We obtain that the condition on the variables $u$, $v$ and $w$ is
\begin{eqnarray}
\label{condition uvw 3}
& & \left( \zeta^{-3} - 7 \delta \right) Z - 10 \delta T
\leq \left| uvw + T \right| < \left( 1 + 7 \delta \right) Z + 10 \delta T \textrm{.}
\end{eqnarray}
We therefore see that this contribution is bounded by
\begin{eqnarray*}
\sum_{\eqref{condition0}, \eqref{condition8}}
\# \left( \mathcal{I} \times \mathcal{J} \times \mathcal{K} \cap \mathbb{Z}^3 \right) & \ll &
\# \left\{ (u,v,w) \in \mathbb{Z}_{\neq 0}^3,
\begin{array}{l}
\eqref{condition uvw 3} \\
uv \ll X / L_2
\end{array}
\right\} \\
& \ll & 
\sum_{uv \ll X / L_2} \left( \frac{\delta Z}{u v} +  \frac{\delta T}{u v} + 1 \right) \\
& \ll & \sum_{v \ll X / L_2} \left( \frac{\delta X^{1 + \varepsilon}}{v} + \frac{X}{L_2 v} \right) \\
& \ll & \delta X^{1 + \varepsilon} + \frac{X \log(X)}{L_2} \textrm{,}
\end{eqnarray*}
since $T \leq X$ and $Z \leq 2X$. Finally, in the case of the condition \eqref{E}, we have
\begin{eqnarray}
\label{condition9}
& & \zeta^{-1} L_1 < V \leq L_1 \textrm{,}
\end{eqnarray}
and thus $\zeta^{-1} L_1 < v \leq \zeta L_1$. We see that the contribution corresponding to this condition is bounded by
\begin{eqnarray*}
\sum_{\eqref{condition0'}, \eqref{condition9}}
\# \left( \mathcal{I} \times \mathcal{J} \times \mathcal{K} \cap \mathbb{Z}^3 \right) & \ll &
\# \left\{ (u,v,w) \in \mathbb{Z}_{\neq 0}^3,
\begin{array}{l}
\zeta^{-1} L_1 < v \leq \zeta L_1 \\
u|w| \ll X / L_1
\end{array}
\right\} \\
& \ll & 
\sum_{u|w| \ll X / L_1} \left( \delta L_1 + 1 \right) \\
& \ll & \sum_{u \ll X / L_1} \left( \frac{\delta X}{u} + \frac{X}{L_1 u} \right) \\
& \ll & \delta X^{1 + \varepsilon} + \frac{X \log(X)}{L_1} \textrm{.}
\end{eqnarray*}
Finally, in a strictly similar way, the contribution of the condition \eqref{F} is seen to be
$\ll \delta X^{1 + \varepsilon} + X \log(X)/L_2$. Writing $1/ \varphi(q) \ll X^{\varepsilon}/q$, we have obtained the estimate
\begin{eqnarray*}
D(\mathcal{S};q,a) - D(\mathcal{S};q) & \ll & X^{\varepsilon} \left( \frac{\delta X}{q} + \frac{E(X,q)}{\delta^3} \right) +
\frac{X \log(X)}{\varphi(q)} \left( \frac1{L_1} + \frac1{L_2} \right) \textrm{.}
\end{eqnarray*}
Recalling the expression \eqref{definition E} of $E(X,q)$, we see that the optimal choice for $\delta$ is
\begin{eqnarray*}
\delta & = & \frac1{q^{1/600}} \left( \frac{q^{2}}{X} \right)^{29/300} \textrm{.}
\end{eqnarray*}
This choice is allowed provided that $q \leq X^{1/2 + 1/230}$. We have finally obtained 
\begin{eqnarray*}
D(\mathcal{S};q,a) - D(\mathcal{S};q) & \ll & \frac{X^{1/2 + \varepsilon}}{q^{1/600}}
\left( \frac{X}{q^2} \right)^{1/2 - 29/300} +
\frac{X \log(X)}{\varphi(q)} \left( \frac1{L_1} + \frac1{L_2} \right) \textrm{.}
\end{eqnarray*}
Averaging this estimate over $a$ coprime to $q$ and using the fact that $D(\mathcal{S};q)$ does not depend on $a$, we see that we can replace $D(\mathcal{S};q)$ by $D^{\ast}(\mathcal{S};q)$ in this estimate, which ends the proof.
\end{proof}

Note that the estimate of lemma \ref{pre lemma tau} is actually true for $q \leq X$ but the error term is no longer better than the trivial error term $X^{1 + \varepsilon} / q$ when $q \geq X^{1/2 + 1/230}$.

It is actually a slightly different version of lemma \ref{pre lemma tau} that we need. Indeed, in our work, the variables $v$ and $w$ have to be coprime. To state the result we require, we define for $b \in \mathbb{Z}_{> 0}$,
\begin{eqnarray*}
N_b(\mathcal{I},\mathcal{J},\mathcal{K}; q,a) & = &
\# \left\{ (u,v,w) \in \mathcal{I} \times \mathcal{J} \times \mathcal{K} \cap \mathbb{Z}^3,
\begin{array}{l}
\gcd(v,bw) = 1 \\
uvw \equiv a \imod{q} \\
\end{array}
\right\} \textrm{,}
\end{eqnarray*}
and, as we can expect,
\begin{eqnarray*}
N_b^{\ast}(\mathcal{I},\mathcal{J},\mathcal{K};q) & = & \frac1{\varphi(q)}
\# \left\{ (u,v,w) \in \mathcal{I} \times \mathcal{J} \times \mathcal{K} \cap \mathbb{Z}^3,
\begin{array}{l}
\gcd(v,bw) = 1 \\
\gcd(uvw,q) = 1 \\
\end{array}
\right\} \textrm{.}
\end{eqnarray*}
Let us also introduce, for $\lambda > 0$, the arithmetic function
\begin{eqnarray}
\label{sigma}
\sigma_{- \lambda}(n) & = & \sum_{k|n} k^{- \lambda} \textrm{.}
\end{eqnarray}
Given lemmas \ref{Fri-Iwa lemma} and \ref{pre lemma tau}, the next two results are straightforward.

\begin{lemma}
\label{Fri-Iwa lemma 2}
Let $\varepsilon > 0$ be fixed. With the same notations as in lemma \ref{Fri-Iwa lemma}, for $q \leq X^{1/2+1/230}$, we have the estimate
\begin{eqnarray*}
N_b(\mathcal{I},\mathcal{J},\mathcal{K};q,a) & = & N_b^{\ast}(\mathcal{I},\mathcal{J},\mathcal{K};q)
+ O \left( \sigma_{- \lambda}(b) X^{\varepsilon} E(X,q) \right) \textrm{,}
\end{eqnarray*}
where $\lambda = 46/75$.
\end{lemma}

\begin{proof}
For $I \subset \mathbb{R}$ and $m \in \mathbb{Z}>0$, we set $I_m = \left\{ t \in \mathbb{R}, m t \in I \right\}$. A first Möbius inversion shows that $N_b(\mathcal{I},\mathcal{J},\mathcal{K};q,a)$ is equal to
\begin{eqnarray*}
\sum_{k \geq 1} \mu(k)
\# \left\{ (u,v',w') \in \mathcal{I} \times \mathcal{J}_k \times \mathcal{K}_k \cap \mathbb{Z}^3,
\begin{array}{l}
\gcd(kv',b) = 1 \\
k^2uv'w' \equiv a \imod{q} \\
\end{array}
\right\} \textrm{,}
\end{eqnarray*}
and a second Möbius inversion shows that $N_b(\mathcal{I},\mathcal{J},\mathcal{K};q,a)$ is also equal to
\begin{eqnarray*}
\sum_{\substack{k \geq 1, \ell |b \\ \gcd(k,b) = 1}} \mu(k) \mu(\ell)
\# \left\{ (u,v'',w') \in \mathcal{I} \times \mathcal{J}_{k \ell} \times \mathcal{K}_k \cap \mathbb{Z}^3,
k^2\ell uv''w' \equiv a \imod{q} \right\} \textrm{.}
\end{eqnarray*}
Furthermore, since $a$ and $q$ are coprime, we can assume that $\gcd(k \ell ,q) = 1$ and thus
$N_b(\mathcal{I},\mathcal{J},\mathcal{K}; q,a)$ is finally equal to
\begin{eqnarray*}
\sum_{\substack{k \geq 1, \ell | b \\ \gcd(k, b) = 1 \\ \gcd(k \ell ,q) = 1}} \! \! \! \! \! \! & & \! \! \! \! \! \! \mu(k)
\mu( \ell )
\# \left\{ (u, v'', w') \in \mathcal{I} \times \mathcal{J}_{k \ell} \times \mathcal{K}_k \cap \mathbb{Z}^3,
uv''w' \equiv k^{-2} \ell^{-1} a \imod{q} \right\} \textrm{.}
\end{eqnarray*}
By the remark immediately following lemma \ref{Fri-Iwa lemma}, there exists a quantity $N_b(\mathcal{I},\mathcal{J},\mathcal{K};q)$ independent of $a$ such that
\begin{eqnarray*}
N_b(\mathcal{I},\mathcal{J},\mathcal{K}; q,a) - N_b(\mathcal{I},\mathcal{J},\mathcal{K};q) & \ll &
\sum_{\substack{k \geq 1 \\ \ell |b}} |\mu(k)| |\mu(\ell)| \left( \frac{X}{k^2 \ell} \right)^{\varepsilon}
E \left( \frac{X}{k^2 \ell},q \right) \\
& \ll & \sigma_{- \lambda}(b) X^{\varepsilon} E(X,q) \textrm{.}
\end{eqnarray*}
Averaging this estimate over $a$ coprime to $q$ proves that $N_b(\mathcal{I},\mathcal{J},\mathcal{K};q)$ can be replaced by
$N_b^{\ast}(\mathcal{I},\mathcal{J},\mathcal{K};q)$ in this estimate, which completes the proof of the lemma.
\end{proof}

We now introduce the following quantities
\begin{eqnarray*}
D_b(\mathcal{S};q,a) & = & \# \left\{ (u,v,w) \in \mathcal{S} \cap \mathbb{Z}^3,
\begin{array}{l}
\gcd(v,bw) = 1 \\
uvw \equiv a \imod{q} \\
\end{array}
\right\} \textrm{,}
\end{eqnarray*}
and
\begin{eqnarray*}
D_b^{\ast}(\mathcal{S};q) & = & \frac1{\varphi(q)} \# \left\{ (u,v,w) \in \mathcal{S} \cap \mathbb{Z}^3,
\begin{array}{l}
\gcd(v,bw) = 1 \\
\gcd(uvw,q) = 1 \\
\end{array}
\right\} \textrm{.}
\end{eqnarray*}

\begin{lemma}
\label{lemma tau}
Let $\varepsilon > 0$ be fixed. Recall the definitions of $\vartheta$ and $\lambda$ respectively given in lemmas \ref{pre lemma tau} and \ref{Fri-Iwa lemma 2}. If $T \leq X$ then for $q \leq X^{1/2 + 1/230}$, we have the estimate
\begin{eqnarray*}
D_b(\mathcal{S};q,a) - D_b^{\ast}(\mathcal{S};q) & \ll &
\sigma_{- \lambda}(b)^{1/4} \frac{X^{1/2 + \varepsilon}}{q^{1/600}} \left( \frac{X}{q^2} \right)^{1/2- \vartheta}
+ \frac{X \log(X)}{\varphi(q)} \left( \frac1{L_1} + \frac1{L_2} \right) \textrm{.}
\end{eqnarray*}
\end{lemma}

\begin{proof}
We proceed as in the proof of lemma \ref{pre lemma tau}. First, we clearly have
\begin{eqnarray*}
D_b(\mathcal{S};q,a) -  \sum_{\mathcal{I} \times \mathcal{J} \times \mathcal{K} \cap \mathbb{Z}^3 \subset \mathcal{S}}
N_b(\mathcal{I},\mathcal{J}, \mathcal{K};q,a) & \ll &
\sum_{\substack{\mathcal{I} \times \mathcal{J} \times \mathcal{K} \cap \mathbb{Z}^3 \nsubseteq \mathcal{S} \\
\mathcal{I} \times \mathcal{J} \times \mathcal{K} \cap \mathbb{Z}^3 \nsubseteq \mathbb{R}^3 \setminus \mathcal{S}}}
N(\mathcal{I},\mathcal{J},\mathcal{K};q,a) \textrm{.}
\end{eqnarray*}
Furthermore, if we define
\begin{eqnarray*}
D_b(\mathcal{S};q) & = & \sum_{\mathcal{I} \times \mathcal{J} \times \mathcal{K} \cap \mathbb{Z}^3 \subset \mathcal{S}}
N_b^{\ast}(\mathcal{I},\mathcal{J},\mathcal{K};q) \textrm{,}
\end{eqnarray*}
we see that $D_b(\mathcal{S};q)$ is independent of $a$ and moreover by lemma \ref{Fri-Iwa lemma 2}, we get
\begin{eqnarray*}
\sum_{\mathcal{I} \times \mathcal{J} \times \mathcal{K} \cap \mathbb{Z}^3 \subset \mathcal{S}}
N_b(\mathcal{I},\mathcal{J},\mathcal{K};q,a) - D_b(\mathcal{S};q) & \ll & \sigma_{- \lambda}(b)
\frac{X^{\varepsilon} E(X,q)}{\delta^3} \textrm{,}
\end{eqnarray*}
where we have used the fact that the number of hyperrectangles $\mathcal{I} \times \mathcal{J} \times \mathcal{K}$ such that
$\mathcal{I} \times \mathcal{J} \times \mathcal{K} \cap \mathbb{Z}^3 \subset \mathcal{S}$ is $\ll X^{\varepsilon} \delta^{-3}$. It is now plain to see that we can conclude exactly as in the proof of lemma \ref{pre lemma tau}.
\end{proof}

Note that the arithmetic function $\sigma_{- \lambda}(b)^{1/4}$ will not intervene in the estimations of our future error terms since it has average order $O(1)$.

It is certainly worth pointing out that using Smith's (see \cite{MR544261}) or Weinstein's (see \cite{MR630958}) version of Deligne's work to bound two-dimensional Kloosterman sums in the most naive way following the reasoning of \cite[Lemma $1$]{3A1}, we would have only obtained that for $q \leq X^{1/2}$, 
\begin{eqnarray*}
D_b(\mathcal{S};q,a) - D_b^{\ast}(\mathcal{S};q) & \ll & \sigma_{- \lambda}(b)^{1/4} X^{1/2 + \varepsilon}
\left( \frac{X}{q^2} \right)^{1/4} + \frac{X \log(X)}{\varphi(q)} \left( \frac1{L_1} + \frac1{L_2} \right) \textrm{,}
\end{eqnarray*}
which would have not been enough to reach our goal. Actually, it would have not even been enough to get an upper bound of the exact order of magnitude for $N_{U,H}(B)$ because this would have required to replace $X^{\varepsilon}$ by $\log(X)$ in the estimate above which seems out of reach with this method.

\subsection{Arithmetic functions}

\label{arithmetic section}

Along the proof of theorem \ref{Manin}, we will meet the following arithmetic functions,

\begin{eqnarray}
\label{varphi}
\varphi^{\ast}(n) & = & \prod_{p|n} \left( 1 - \frac1{p} \right) \textrm{,} \\
\label{varphi+}
\varphi^{+}(n) & = & \prod_{p|n} \left( 1 + \frac1{p} \right)^{-1} \textrm{,} \\
\label{varphi times}
\varphi^{\times}(n) & = & \prod_{p|n} \left( 1 + \frac1{p} \right) \left( 1 - \frac1{p} \right)^{-1}
\left( 1 + \frac{2}{p} - \frac1{p^2} \right)^{-1} \textrm{,}
\end{eqnarray}
and also, for $a, b \in \mathbb{Z}_{\geq 1}$,
\begin{eqnarray*}
\psi_{a,b}(n) & = &
\begin{cases}
\varphi^{\ast}(n) \varphi^{\ast}(\gcd(a,n))^{-1} & \textrm{ if } \gcd(n,b) = 1 \textrm{,} \\
0 & \textrm{ otherwise,}
\end{cases}
\end{eqnarray*}
and
\begin{eqnarray*}
\psi_{a,b}'(n) & = &
\begin{cases}
\varphi^{\ast}(n)^2 \varphi^{\ast}(\gcd(a,n))^{-2} \varphi^{+}(n) \varphi^{+}(\gcd(a,n))^{-1} & \textrm{ if } \gcd(n,b) = 1 \textrm{,} \\
0 & \textrm{ otherwise.}
\end{cases}
\end{eqnarray*}
The following lemma immediately follows from \cite[Lemma $3$.$2$]{MR2543667}.

\begin{lemma}
\label{arithmetic preliminary 1}
Let $0 \leq t_1 < t_2$ and set $I = [t_1,t_2]$. Let $g : \mathbb{R}_{> 0} \to \mathbb{R}$ be a function having a piecewise continuous derivative on $I$ whose sign changes at most $R_g(I)$ times on $I$. We have
\begin{eqnarray*}
\sum_{n \in I \cap \mathbb{Z}_{>0}} \psi_{a,b}(n) g(n) & = & \zeta(2)^{-1} \Psi(a,b) \int_I g(t) \D t +
O \left( 2^{\omega(b)} \log \left( 1 + t_2 \right) M_I(g) \right) \textrm{,}
\end{eqnarray*}
where
\begin{eqnarray*}
\Psi(a,b) & = & \varphi^{\ast}(b) \frac{\varphi^{+}(ab)}{\varphi^{\ast}(ab)} \textrm{,} \\
M_I(g) & = & (1 + R_g(I)) \sup_{t \in I \cap \mathbb{R}_{> 0}} |g(t)| \textrm{.}
\end{eqnarray*}
\end{lemma}

Our goal in this section is to get a similar result for $\psi_{a,b}'$. Recall the definition \eqref{sigma} of $\sigma_{- \lambda}$. The proof of the following lemma is straightforward.

\begin{lemma}
\label{arithmetic preliminary 2}
Let $0 < \gamma \leq 1$ be fixed. We have the estimate
\begin{eqnarray*}
\sum_{n \leq X} \psi_{a,b}'(n)& = & \Xi \Psi'(a,b) X + O_{\gamma} \left( \sigma_{- \gamma/2}(b) X^{\gamma} \right) \textrm{,}
\end{eqnarray*}
where
\begin{eqnarray}
\label{Xi}
\Xi & = & \prod_p \varphi^{\times}(p)^{-1} \textrm{,} \\
\notag
\Psi'(a,b) & = & \varphi^{\ast}(b) \varphi^{\times}(ab)  \textrm{.}
\end{eqnarray}
\end{lemma}

\begin{proof}
Writing $\psi_{a,b}' = (\psi_{a,b}' \ast \mu) \ast 1$, we obtain
\begin{eqnarray}
\label{convolution}
\sum_{n \leq X} \psi_{a,b}'(n) & = & \sum_{d = 1}^{+ \infty} (\psi_{a,b}' \ast \mu)(d) \left[ \frac{X}{d} \right] \textrm{.}
\end{eqnarray}
Let us calculate the Dirichlet convolution of $\psi_{a,b}'$ with the Möbius function $\mu$. We immediately get
\begin{eqnarray*}
(\psi_{a,b}' \ast \mu)(n) & = &
\prod_{p^\nu \parallel n} \left( \psi_{a,b}' \left( p^\nu \right) - \psi_{a,b}' \left( p^{\nu - 1} \right) \right) \textrm{.}
\end{eqnarray*}
We have $\psi_{a,b}'(1) = 1$ and for any $\nu \geq 1$,
$$\psi_{a,b}' \left( p^\nu \right) = \psi_{a,b}'(p) =
\begin{cases}
\left( 1 - 1/p \right)^{2} \left( 1 + 1/p \right)^{-1} & \textrm{ if } p \nmid ab \textrm{,} \\
1 & \textrm{ if } p|a, p \nmid b \textrm{,} \\
0 & \textrm{ if } p|b \textrm{.}
\end{cases}
$$
A short calculation thereby shows that if $\gcd(n,a)|b$ then
\begin{eqnarray*}
(\psi_{a,b}' \ast \mu)(n) & = &
\mu(n) \prod_{p|n, p\nmid b} \left( 1- \left( 1 - \frac1{p} \right)^{2} \left( 1 + \frac1{p} \right)^{-1} \right) \\
& = & \mu(n) 3^{\omega(n) - \omega(\gcd(n,b))} \frac{\gcd(n,b)}{n}
\prod_{p|n, p\nmid b} \left( 1 + \frac1{p} \right)^{-1} \left( 1 - \frac1{3p} \right) \textrm{,}
\end{eqnarray*}
and $(\psi_{a,b}' \ast \mu)(n) = 0$ otherwise. Let $0 < \gamma \leq 1$ be fixed. Let us use the elementary estimate
$[t] = t + O(t^{\gamma})$ for $t = X /d$ in \eqref{convolution}. Since
\begin{eqnarray*}
|(\psi_{a,b}' \ast \mu)(n)| & \leq & 3^{\omega(n)} \frac{\gcd(n,b)}{n} \textrm{,}
\end{eqnarray*}
we easily get
\begin{eqnarray*}
\sum_{d = 1}^{+ \infty} \frac{| (\psi_{a,b}'  \ast \mu)(d) |}{d^{\gamma}} & \leq & \sum_{k|b} \sum_{\ell = 1}^{+ \infty}
3^{\omega(k\ell)} \frac{k}{(k\ell)^{1 + \gamma}} \\
& \ll & \sum_{k|b} \frac{3^{\omega(k)}}{k^{\gamma}} \\
& \ll & \sigma_{- \gamma/2}(b) \textrm{,}
\end{eqnarray*}
where we have used $3^{\omega(k)} \ll k^{\gamma / 2}$. We have proved the estimate
\begin{eqnarray*}
\sum_{n \leq X} \psi_{a,b}'(n) & = & X \sum_{d = 1}^{+ \infty} \frac{(\psi_{a,b}' \ast \mu)(d)}{d}
+ O \left( \sigma_{- \gamma/2}(b) X^{\gamma} \right) \textrm{.}
\end{eqnarray*}
Finally, a straigthforward calculation yields
\begin{eqnarray*}
\sum_{d = 1}^{+ \infty} \frac{(\psi_{a,b}'  \ast \mu)(d)}{d} & = & \prod_{p|b} \left( 1 - \frac1{p} \right)
\prod_{p \nmid ab}\left( 1 + \frac1{p} \right)^{-1} \left( 1 - \frac1{p} \right)
\left( 1 + \frac{2}{p} - \frac1{p^2} \right) \textrm{,}
\end{eqnarray*}
which completes the proof.
\end{proof}

Using partial summation and the estimate of lemma \ref{arithmetic preliminary 2} exactly as in the proof of
\cite[Lemma $6$]{3A1}, we get the following result.

\begin{lemma}
\label{arithmetic preliminary 2'}
Let $0 < \gamma \leq 1$ be fixed. With the notations of lemma \ref{arithmetic preliminary 1}, we have
\begin{eqnarray*}
\sum_{n \in I \cap \mathbb{Z}_{>0}} \psi_{a,b}'(n) g(n) & = & \Xi \Psi'(a,b) \int_I g(t) \D t +
O_{\gamma} \left( \sigma_{- \gamma/2}(b) t_2^{\gamma} M_I(g) \right) \textrm{.}
\end{eqnarray*}
\end{lemma}

\section{The universal torsor}

\label{torsor section}

In this section we derive a bijection between the set of rational points of bounded height on $U$ we aim to investigate and a certain set of integral points lying on the hypersurface defined in the introduction. Our choice of notation might seem awkward but it is directed by our wish to follow the notation used by Derenthal in \cite{D-hyper}. It is immediate to notice that if
$(x_0:x_1:x_2:x_3) \in V(\mathbb{Q})$ then $(x_0:x_1:x_2:x_3) \in U(\mathbb{Q})$ if and only if $x_0 x_1 x_2 x_3 \neq 0$. Moreover, since $\mathbf{x} = - \mathbf{x} \in \mathbb{P}^3$, we can assume that $x_1 > 0$. We thus let
$(x_0,x_1,x_2,x_3) \in \mathbb{Z}_{\neq 0} \times \mathbb{Z}_{>0} \times \mathbb{Z}_{\neq 0}^2$ be such that
$\gcd(x_0,x_1,x_2,x_3) = 1$ and
\begin{eqnarray*}
x_3^2 (x_1 + x_3) + x_0 x_1 x_2 & = & 0 \textrm{.}
\end{eqnarray*}
Note that this equation together with the condition $\gcd(x_0,x_1,x_2,x_3) = 1$ imply that we actually have $\gcd(x_0,x_1,x_2) = 1$. Let $\eta_1 = \gcd(x_0,x_2,x_3)$, $y_0 = \gcd(x_1,x_2,x_3)$ and $y_2 = \gcd(x_0,x_1,x_3)$. Given that $\gcd(x_0,x_1,x_2) = 1$, we can write
\begin{eqnarray*}
x_0 & = & \eta_1 y_2 x_0' \textrm{,} \\
x_1 & = & y_0 y_2 x_1' \textrm{,} \\
x_2 & = & \eta_1 y_0 x_2' \textrm{,} \\
x_3 & = & \eta_1 y_0 y_2 x_3' \textrm{,}
\end{eqnarray*}
where $\eta_1, y_0, y_2, x_1' > 0$ and $\gcd(y_i,x_i') = 1$ for $i \in \{1, 2\}$, $\gcd(\eta_1,y_0y_2x_1') = 1$, $\gcd(y_0,y_2) = 1$ and $\gcd(x_i',x_j',x_3') = 1$ for $i,j \in \{0,1,2 \}$ and $i \neq j$. The equation becomes
\begin{eqnarray*}
y_0 y_2 x_3'^2 (x_1' + \eta_1 x_3') + x_0' x_1' x_2' & = & 0 \textrm{.}
\end{eqnarray*}
Letting $\eta_2 = \gcd(y_0,x_2')$, we can write $y_0 = \eta_2 \eta_3$ and $x_2' = \eta_2 x_2''$ where $\eta_2, \eta_3 > 0$ and $\gcd(\eta_3,x_2'') = 1$. We then see that $\eta_3|x_1'$ and thus we can write $x_1' = \eta_3 x_1''$ for some $x_1'' > 0$. Similarly, if we set $\eta_4 = \gcd(y_2,x_0')$, we can write $y_2 = \eta_4 \eta_5$ and $x_0' = \eta_4 x_0''$ for
$\eta_4, \eta_5 > 0$ satisfying $\gcd(\eta_5,x_0'') = 1$. We then notice that $\eta_5|x_1''$ and we write $x_1'' = \eta_5 z_1$ for some $z_1 > 0$. This leads to the equation
\begin{eqnarray*}
x_3'^2 (\eta_3 \eta_5 z_1 + \eta_1 x_3') + x_0'' z_1 x_2'' & = & 0 \textrm{.}
\end{eqnarray*}
Now, let $\eta_6 = \gcd(x_0'',x_3')$, we see that $\eta_6^2|x_0''$ and thereby we can write $x_0'' = \eta_6^2 \eta_9$ and
$x_3' = \eta_6 x_3''$ where $\eta_6 > 0$ and $\gcd(\eta_9,x_3'') = 1$. We can also set $\eta_8 = \pm \gcd(x_2'',x_3'')$ and write $x_2'' = \eta_8^2 \eta_{10}$ and $x_3'' = \eta_7 \eta_8$ for some $\eta_7 > 0$ such that $\gcd(\eta_{10},\eta_7) = 1$. We get the equation
\begin{eqnarray*}
\eta_7^2 (\eta_3 \eta_5 z_1 + \eta_1 \eta_6 \eta_7 \eta_8) + \eta_9 \eta_{10} z_1 & = & 0 \textrm{.}
\end{eqnarray*}
Writing it as
\begin{eqnarray*}
\eta_1 \eta_6 \eta_7^3 \eta_8 + z_1 \left( \eta_3 \eta_5 \eta_7^2 + \eta_9 \eta_{10} \right) & = & 0 \textrm{,}
\end{eqnarray*}
and reminding that $\gcd(z_1, \eta_1 \eta_6 \eta_8) = 1$ and $\gcd(\eta_7, \eta_9 \eta_{10}) = 1$, we see that we have at once
$z_1|\eta_7^3$ and $\eta_7^3|z_1$. Consequently, since $z_1 > 0$, we have $z_1 = \eta_7^3$ and dividing by $\eta_7^3$, we finally obtain the desired equation
\begin{eqnarray}
\label{torsor}
\eta_1 \eta_6 \eta_8 + \eta_3 \eta_5 \eta_7^2 + \eta_9 \eta_{10} & = & 0 \textrm{.}
\end{eqnarray}
It is an easy task to check that the coprimality conditions we have derived along our investigation imply that the two monomials $\eta_3 \eta_5 \eta_7^2$ and $\eta_9 \eta_{10}$ are coprime. Thereby, given the equation \eqref{torsor}, we deduce that the three monomials $\eta_3 \eta_5 \eta_7^2$, $\eta_1 \eta_6 \eta_8$ and $\eta_9 \eta_{10}$ are actually pairwise coprime. With this remark in mind, a little thought reveals that all the coprimality conditions can be rewritten as
\begin{eqnarray}
\label{gcd1}
& & \gcd(\eta_9, \eta_1 \eta_2 \eta_3 \eta_5 \eta_6 \eta_7 \eta_8) = 1 \textrm{,} \\
\label{gcd2}
& & \gcd(\eta_1, \eta_2 \eta_3 \eta_4 \eta_5 \eta_7 \eta_{10}) = 1 \textrm{,} \\
\label{gcd3}
& & \gcd(\eta_6, \eta_2 \eta_3 \eta_5 \eta_7 \eta_{10}) = 1 \textrm{,} \\
\label{gcd4}
& & \gcd(\eta_8, \eta_3 \eta_4 \eta_5 \eta_6 \eta_7 \eta_{10}) = 1 \textrm{,} \\
\label{gcd5}
& & \gcd(\eta_7, \eta_2 \eta_4 \eta_{10}) = 1 \textrm{,} \\
\label{gcd6}
& & \gcd(\eta_2 \eta_3, \eta_4 \eta_5) = 1 \textrm{,} \\
\label{gcd7}
& & \gcd(\eta_{10}, \eta_3 \eta_4 \eta_5) = 1 \textrm{.}
\end{eqnarray}
Let $\mathcal{T}(B)$ be the number of $(\eta_1, \dots, \eta_{10}) \in \mathbb{Z}_{>0}^7 \times \mathbb{Z}_{\neq 0}^3$ satisfying the equation \eqref{torsor}, the coprimality conditions \eqref{gcd1}, \eqref{gcd2}, \eqref{gcd3}, \eqref{gcd4}, \eqref{gcd5}, \eqref{gcd6}, \eqref{gcd7} and the height conditions
\begin{eqnarray}
\label{height1}
\eta_1 \eta_4^2 \eta_5 \eta_6^2 |\eta_9| & \leq & B \textrm{,} \\
\label{height2}
\eta_2 \eta_3^2 \eta_4 \eta_5^2 \eta_7^3 & \leq & B \textrm{,} \\
\label{height3}
\eta_1 \eta_2^2 \eta_3 \eta_8^2 |\eta_{10}| & \leq & B \textrm{,} \\
\label{height4}
\eta_1 \eta_2 \eta_3 \eta_4 \eta_5 \eta_6 \eta_7 |\eta_8| & \leq & B \textrm{.}
\end{eqnarray}
We sum up the fruit of our investigation in the following lemma.

\begin{lemma}
\label{T}
We have the equality
\begin{eqnarray*}
N_{U,H}(B) & = & \# \mathcal{T}(B) \textrm{.}
\end{eqnarray*}
\end{lemma}

\section{Calculation of Peyre's constant}

The constant $c_{V,H}$ appearing in the statement of theorem \ref{Manin} and whose conjectural interpretation is due to Peyre (see \cite{MR1340296}) is expected to be equal to the following product
\begin{eqnarray*}
c_{V,H} & = & \alpha(\widetilde{V}) \beta(\widetilde{V}) \omega_H(\widetilde{V}) \textrm{,}
\end{eqnarray*}
where we recall that $\widetilde{V}$ denotes the minimal desingularization of $V$. This section is devoted to the investigation of these three quantities. Let $\Pic(\widetilde{V})_{\mathbb{R}} = \Pic(\widetilde{V}) \otimes_{\mathbb{Z}} \mathbb{R}$ and let
$\Lambda_{\eff}(\widetilde{V})$ be the effective cone of $\widetilde{V}$, that is the cone generated by the classes of effective divisors in $\Pic(\widetilde{V})_{\mathbb{R}}$. Let also $\Lambda^{\vee}_{\eff}(\widetilde{V})$ be the dual of
$\Lambda_{\eff}(\widetilde{V})$ with respect to the intersection form. Finally let $-K_{\widetilde{V}}$ be the anticanonical divisor of $\widetilde{V}$. By definition
\begin{eqnarray*}
\alpha(\widetilde{V}) & = & \vol ( P(\widetilde{V}) ) \textrm{,}
\end{eqnarray*}
where
\begin{eqnarray*}
P(\widetilde{V}) & = & \{ \mathbf{x} \in \Lambda^{\vee}_{\eff}(\widetilde{V}), ( -K_{\widetilde{V}}, \mathbf{x} ) = 1 \} \textrm{,}
\end{eqnarray*}
and where the measure on the hyperplane
\begin{eqnarray*}
& & \{ \mathbf{x} \in \Pic(\widetilde{V})_{\mathbb{R}}^{\vee}, ( -K_{\widetilde{V}}, \mathbf{x} ) = 1 \}
\end{eqnarray*}
is defined by the $6$-form $\D \mathbf{x}$ such that $\D \mathbf{x} \wedge \D \omega = \D \mathbf{y}$, where $\D \omega$ is the linear form defined by $-K_{\widetilde{V}}$ on $\Pic(\widetilde{V})_{\mathbb{R}}^{\vee}$ and $\D \mathbf{y}$ is the natural Lebesgue measure on $\Pic(\widetilde{V})_{\mathbb{R}}^{\vee}$. This constant can be easily computed using the work of Derenthal, Joyce and Teitler \cite[Theorem $1$.$3$]{MR2377367} and we find
\begin{eqnarray*}
\alpha(\widetilde{V}) & = & \frac1{120} \cdot \frac1{\# W(\mathbf{A}_2)^2 \ \# W(\mathbf{A}_1)} \\
& = & \frac1{8640} \textrm{,}
\end{eqnarray*}
where $W(\mathbf{A}_n)$ stands for the Weyl group associated to the Dynkin diagram of the singularity $\mathbf{A}_n$ and where we have used the fact that $W(\mathbf{A}_n) \simeq \mathfrak{S}_{n+1}$ and therefore $\# W(\mathbf{A}_n) = (n+1)!$. Note that the value given in \cite[Table $6$]{MR2318651} was misprinted. In addition, $\beta(\widetilde{V})$ is defined by
\begin{eqnarray*}
\beta(\widetilde{V}) & = & \# H^1 ( \Gal (\overline{\mathbb{Q}}/\mathbb{Q}),\Pic_{\overline{\mathbb{Q}}}(\widetilde{V})) \textrm{,}
\end{eqnarray*}
and here $\beta(\widetilde{V}) = 1$ since $V$ is split over $\mathbb{Q}$. Finally, again in the particular case of a surface split over $\mathbb{Q}$, $\omega_H(\widetilde{V})$ is defined by 
\begin{eqnarray*}
\omega_H(\widetilde{V}) & = &
\lim_{s \rightarrow 1} \left( (s-1)^{\rho} \zeta(s)^{\rho} \right) \omega_{\infty}
\prod_p \left( 1 - \frac1{p} \right)^{\rho} \omega_p \\
& = & \omega_{\infty} \prod_p \left( 1 - \frac1{p} \right)^{7} \omega_p \textrm{,}
\end{eqnarray*}
where we recall that $\rho = \rho_{\widetilde{V}}$ denotes the rank of the Picard group of $\widetilde{V}$ and where $\omega_{\infty}$ and $\omega_p$ are respectively the archimedean and $p$-adic densities. Let $\mathbf{x} = (x_0,x_1,x_2,x_3)$ and
$f(\mathbf{x}) = x_3^2 (x_1 + x_3) + x_0 x_1 x_2$. The densities $\omega_p$ are given by
\begin{eqnarray*}
\omega_p & = & \lim_{n \rightarrow + \infty} \frac{ \# \left\{ \mathbf{x} \imod{p^n}, f(\mathbf{x}) \equiv 0 \imod{p^n}  \right\} }{p^{n(\dim(\widetilde{V})+1)}} \textrm{.}
\end{eqnarray*}
By a result of Loughran \cite[Lemma 2.3]{Loughran}, we have
\begin{eqnarray*}
\omega_p & = & 1 + \frac{7}{p} + \frac1{p^2} \textrm{.}
\end{eqnarray*}
Let us express $\omega_{\infty}$. We parametrize the points of $V$ with $x_1$, $x_2$ and $x_3$. We have
\begin{eqnarray*}
\frac{\partial f}{\partial x_0}(\mathbf{x}) & = & x_1 x_2 \textrm{,}
\end{eqnarray*}
and since $\mathbf{x} = - \mathbf{x} \in \mathbb{P}^3$, we get
\begin{eqnarray*}
\omega_{\infty} & = & 2 \int \int \int_{f(\mathbf{x}) = 0, 0 < |x_0|, x_1, x_2, |x_3| \leq 1} 
\left(\frac{\partial f}{\partial x_0}(\mathbf{x}) \right)^{-1} \D x_1 \D x_2 \D x_3 \\
& = & 2 \int \int \int_{0 < x_3^2 \left| x_1 + x_3 \right| /x_1 x_2, x_1, x_2, |x_3| \leq 1}
\frac{\D x_1 \D x_2 \D x_3}{x_1x_2} \textrm{.}
\end{eqnarray*}
Define the real-valued function
\begin{eqnarray}
\label{equation h}
h & : & (t_8,t_6,t_7) \mapsto \max \left\{ t_6^2 \left| t_7^2 + t_6t_8 \right|, t_7, |t_8|, t_6 t_7 |t_8| \right\} \textrm{.}
\end{eqnarray}
The change of variables given by $x_1 = t_7^3$, $x_2 = t_8^2$ and $x_3 = t_6 t_7 t_8$ yields the expression
\begin{eqnarray}
\label{omega}
\omega_{\infty} & = & 12 \int \int \int_{t_6, t_7 > 0, h(t_8,t_6,t_7) \leq 1} \D t_8 \D t_6 \D t_7 \textrm{.}
\end{eqnarray}

\section{Proof of the main theorem}

\subsection{Restriction of the domain}

The following lemma proves that we can assume that certain variables are greater in absolute value than a fixed power of $\log(B)$.

\begin{lemma}
\label{lemmalog}
Let $\mathcal{M}(B)$ be the overall contribution to $N_{U,H}(B)$ coming from the $(\eta_1, \dots ,\eta_{10}) \in \mathcal{T}(B)$ such that for a certain $i \neq 9, 10$, $\eta_i$ is subject to the condition $|\eta_i| \leq \log(B)^A$ where $A > 0$ is any fixed constant. We have the estimate
\begin{eqnarray*}
\mathcal{M}(B) & \ll_A & B \log(B)^5 \log(\log(B)) \textrm{.}
\end{eqnarray*}
\end{lemma}

We first need to prove the following result.

\begin{lemma}
\label{lemma bound}
Let $K_1, K_3, K_5, K_6 \ldots, K_{10} \geq 1/2$ and let $K$ be a quantity such that
$K \geq \max \left(K_1 K_6 K_8, K_3 K_5 K_7^2 \right)$. Define $M = M(K_1, K_3, K_5, K_6, \ldots, K_{10})$ be the number of
$(m_1, m_3, m_5, m_6, \ldots, m_{10}) \in \mathbb{Z}^8$ such that $K_i < |m_i| \leq 2K_i$ for $i = 1, 3$ and $5 \leq i \leq 10$, $\gcd(m_1m_6m_8,m_9m_{10}) = 1$ and finally
\begin{equation}
\label{equation lemma}
m_1m_6m_8 + m_3m_5m_7^2 + m_9m_{10} = 0 \textrm{.} 
\end{equation}
For any fixed $\varepsilon > 0$, we have the estimate
\begin{eqnarray*}
M & \ll & K_3 K_5 K_7 \left( \min(K_1 K_6 K_8, K_9 K_{10}) +
K^{1 - \vartheta + \varepsilon} (K_9 K_{10})^{\vartheta - 1/1200} \right) \textrm{,}
\end{eqnarray*}
where $\vartheta$ is defined in lemma \ref{pre lemma tau}.
\end{lemma}

It is instructive to compare lemma \ref{lemma bound} with \cite[Lemmas $3$, $4$]{MR2075628} which have been obtained by Heath-Brown in order to deal with the case of Cayley's cubic surface. In particular, it is worth pointing out that in the application of lemma \ref{lemma bound} to prove lemma \ref{lemmalog}, the presence of the $\min$ in the right-hand side of the bound for $M$ is crucial.

\begin{proof}
Note that the equation \eqref{equation lemma} implies $K_9 K_{10} \ll K$. We assume by symmetry that $|m_9| \geq |m_{10}|$ and without loss of generality that $m_{10} > 0$ and we start by counting the number of $m_9$, $m_1$, $m_6$ and $m_8$. The idea is to see the equation \eqref{equation lemma} as a congruence modulo $m_{10}$. Set $a = - m_3m_5m_7^2$ and $q = m_{10}$. Let
\begin{eqnarray}
\label{definition S}
\mathcal{S} & = & \left\{ (x_1, x_6, x_8) \in \mathbb{R}^3,
\begin{array}{l}
K_i < |x_i| \leq 2 K_i, i \in \{1, 6, 8 \} \\
q K_9 < |x_1x_6x_8 - a| \leq 2 q K_9
\end{array}
\right\} \textrm{.}
\end{eqnarray}
The number of $m_9$, $m_1$, $m_6$ and $m_8$ we want to estimate is less than the quantity $D(\mathcal{S};q,a)$ defined by
\begin{eqnarray*}
D(\mathcal{S};q,a) & = & \# \left\{ (m_1,m_6,m_8) \in \mathcal{S} \cap \mathbb{Z}^3, m_1m_6m_8 \equiv a \imod{q} \right\} \textrm{.}
\end{eqnarray*}
We now proceed to investigate this quantity exactly as in the proof of lemma \ref{pre lemma tau}. Let $\delta$ be a parameter such that $K^{-1/3} \leq \delta \leq 1$ and $\zeta = 1 + \delta$. Let also $U_1$, $U_6$ and $U_8$ be variables running over the set $\{ \pm \zeta^{n}, n \in \mathbb{Z}_{\geq 0} \}$ and $\mathcal{I}_i = ]U_i K_i,\zeta U_i K_i]$ if $U_i > 0$ and $\mathcal{I}_i = [\zeta U_i K_i, U_i K_i[$ if $U_i < 0$ for $i \in \{1, 6, 8 \}$. Recall the definitions \eqref{N} and \eqref{Nast} of $N(\mathcal{I}, \mathcal{J}, \mathcal{K}; q, a)$ and $N^{\ast}(\mathcal{I}, \mathcal{J}, \mathcal{K}; q)$. We have
\begin{eqnarray*}
D(\mathcal{S};q,a) - \sum_{\mathcal{I}_1 \times \mathcal{I}_6 \times \mathcal{I}_8 \cap \mathbb{Z}^3 \subset \mathcal{S}} \! \! \! \!
N(\mathcal{I}_1, \mathcal{I}_6, \mathcal{I}_8; q, a) & \ll &
\sum_{\substack{\mathcal{I}_1 \times \mathcal{I}_6 \times \mathcal{I}_8 \cap \mathbb{Z}^3 \nsubseteq \mathcal{S} \\ \mathcal{I}_1 \times \mathcal{I}_6 \times \mathcal{I}_8 \cap \mathbb{Z}^3 \nsubseteq \mathbb{R}^3 \setminus \mathcal{S}}} \! \! \! \!
N(\mathcal{I}_1,\mathcal{I}_6,\mathcal{I}_8;q,a) \textrm{.}
\end{eqnarray*}
We define the quantity
\begin{eqnarray*}
D(\mathcal{S};q) & = & \sum_{\mathcal{I}_1 \times \mathcal{I}_6 \times \mathcal{I}_8 \cap \mathbb{Z}^3 \subset \mathcal{S}}
N^{\ast}(\mathcal{I}_1,\mathcal{I}_6,\mathcal{I}_8;q) \textrm{.}
\end{eqnarray*}
We note that since $N^{\ast}(\mathcal{I}_1,\mathcal{I}_6,\mathcal{I}_8;q)$ is independent of $a$, so is $D(\mathcal{S};q)$. Recall the definition \eqref{definition E} of $E(K,q)$. Using lemma \ref{Fri-Iwa lemma} exactly as in the proof of lemma \ref{pre lemma tau}, we get
\begin{eqnarray*}
\sum_{\mathcal{I}_1 \times \mathcal{I}_6 \times \mathcal{I}_8 \cap \mathbb{Z}^3 \subset \mathcal{S}}
N(\mathcal{I}_1,\mathcal{I}_6,\mathcal{I}_8;q,a) - D(\mathcal{S};q) & \ll & \frac{K^{\varepsilon} E(K,q)}{\delta^3} \textrm{.}
\end{eqnarray*}
We have therefore proved that
\begin{eqnarray*}
D(\mathcal{S};q,a) - D(\mathcal{S};q) & \ll &
\sum_{\substack{\mathcal{I}_1 \times \mathcal{I}_6 \times \mathcal{I}_8 \cap \mathbb{Z}^3 \nsubseteq \mathcal{S} \\
\mathcal{I}_1 \times \mathcal{I}_6 \times \mathcal{I}_8 \cap \mathbb{Z}^3 \nsubseteq \mathbb{R}^3 \setminus \mathcal{S}}}
N(\mathcal{I}_1,\mathcal{I}_6,\mathcal{I}_8;q,a) + \frac{K^{\varepsilon} E(K,q)}{\delta^3} \textrm{.}
\end{eqnarray*}
Using the bound \eqref{bound} for
$N(\mathcal{I}_1,\mathcal{I}_6,\mathcal{I}_8;q,a)$ as in the proof of lemma
\ref{pre lemma tau}, we obtain
\begin{eqnarray*}
D(\mathcal{S};q,a) - D(\mathcal{S};q) & \ll & \frac1{\varphi(q)} \! \!
\sum_{\substack{\mathcal{I}_1 \times \mathcal{I}_6 \times \mathcal{I}_8 \cap \mathbb{Z}^3 \nsubseteq \mathcal{S} \\ \mathcal{I}_1 \times \mathcal{I}_6 \times \mathcal{I}_8 \cap \mathbb{Z}^3 \nsubseteq \mathbb{R}^3 \setminus \mathcal{S}}}
\! \! \! \! \# \left( \mathcal{I}_1 \times \mathcal{I}_6 \times \mathcal{I}_8 \cap \mathbb{Z}^3 \right)
+ \frac{K^{\varepsilon} E(K,q)}{\delta^3} \textrm{.}
\end{eqnarray*}
The hyperrectangles $\mathcal{I}_1 \times \mathcal{I}_6 \times \mathcal{I}_8$ subject to the conditions
$\mathcal{I}_1 \times \mathcal{I}_6 \times \mathcal{I}_8 \cap \mathbb{Z}^3 \nsubseteq \mathcal{S}$ and 
$\mathcal{I}_1 \times \mathcal{I}_6 \times \mathcal{I}_8 \cap \mathbb{Z}^3 \nsubseteq \mathbb{R}^3 \setminus \mathcal{S}$ are those for which $(\zeta^{s_1} U_1 K_1, \zeta^{s_6} U_6 K_6, \zeta^{s_8} U_8 K_8) \in \mathcal{S} \cap \mathbb{Z}^3$ and
$(\zeta^{t_1} U_1 K_1, \zeta^{t_6} U_6 K_6, \zeta^{t_8} U_8 K_8) \in \mathbb{Z}^3 \setminus \mathcal{S}$ for some triples
$(s_1,s_6,s_8) \in ]0,1]^3$ and $(t_1,t_6,t_8) \in ]0,1]^3$. Since we can assume without loss of generality that $2$ is an integer power of $\zeta$, for these hyperrectangles, we have either the two conditions
\begin{eqnarray}
\label{condition U_iK_i'}
q K_9 & < & \left| \zeta^{s_1 + s_6 + s_8} U_1 U_6 U_8 K_1 K_6 K_8 - a \right| \textrm{,} \\
\label{condition U_iK_i}
q K_9 & \geq & \left| \zeta^{t_1 + t_6 + t_8} U_1 U_6 U_8 K_1 K_6 K_8 - a \right| \textrm{,}
\end{eqnarray}
or the two conditions
\begin{eqnarray*}
\left| \zeta^{s_1 + s_6 + s_8} U_1 U_6 U_8 K_1 K_6 K_8 - a \right| & \leq & 2 q K_9 \textrm{,} \\
\left| \zeta^{t_1 + t_6 + t_8} U_1 U_6 U_8 K_1 K_6 K_8 - a \right| & > & 2 q K_9 \textrm{.}
\end{eqnarray*}
The treatments of these two cases are identical so we only deal with the first. We see that the conditions \eqref{condition U_iK_i'} and \eqref{condition U_iK_i} imply
\begin{eqnarray}
\label{condition U_i}
& & \zeta^{-3} q K_9 - \left( 1 - \zeta^{-3} \right) |a| < \left| U_1 U_6 U_8 K_1 K_6 K_8 - a \right| \leq
q K_9 + \left( 1 - \zeta^{-3} \right) |a| \textrm{.}
\end{eqnarray}
Going back to the variables $m_1$, $m_6$ and $m_8$, we easily see that
\begin{eqnarray*}
\big| \left| U_1 U_6 U_8 K_1 K_6 K_8 - a \right| - \left| m_1 m_6 m_8 - a \right| \big| & \leq &
\left| m_1 m_6 m_8 - U_1 U_6 U_8 K_1 K_6 K_8 \right| \\
& \leq & 7 \delta U_1 U_6 U_8 K_1 K_6 K_8 \\
& \leq & 7 \delta (q K_9 + |a|) \textrm{,}
\end{eqnarray*}
using the inequality \eqref{condition U_iK_i}. Since $1 - \zeta^{-3} \leq 3 \delta$, we therefore get
\begin{eqnarray}
\label{condition m_i}
& & \left( \zeta^{-3} - 7 \delta \right) q K_9 - 10 \delta |a| < \left| m_1 m_6 m_8 - a \right| \leq
\left( 1 + 7 \delta \right) q K_9 + 10 \delta |a| \textrm{.}
\end{eqnarray}
We can assume by symmetry that $K_1 \geq K_6,K_8$ and thus $K_6 K_8 \ll K^{2/3}$. Summing first over $m_1$, we see that the error we want to bound is 
\begin{eqnarray*}
\sum_{\eqref{condition U_i}} \# \left( \mathcal{I}_1 \times \mathcal{I}_6 \times \mathcal{I}_8 \cap \mathbb{Z}^3 \right) & \ll &
\# \left\{ (m_1,m_6,m_8) \in \mathbb{Z}^3,
\begin{array}{l}
\eqref{condition m_i} \\
K_6 < |m_6| \leq 2 K_6 \\
K_8 < |m_8| \leq 2 K_8
\end{array}
\right\} \\
& \ll & \sum_{\substack{K_6 < |m_6| \leq 2 K_6 \\ K_8 < |m_8| \leq 2 K_8}}
\left( \frac{\delta q K_9}{|m_6 m_8|} + \frac{\delta |a|}{|m_6 m_8|} + 1 \right) \\
& \ll & \delta q K_9 + \delta |a| + K_6 K_8 \\
& \ll & \delta K \textrm{,}
\end{eqnarray*}
since $q K_9 \ll K$, $|a| \ll K$ and $K_6 K_8 \ll K^{2/3}$, $K^{-1/3} \leq \delta$. We therefore obtain
\begin{eqnarray*}
D(\mathcal{S};q,a) - D(\mathcal{S};q) & \ll & \frac{\delta K}{\varphi(q)} + \frac{K^{\varepsilon} E(K,q)}{\delta^3} \\
& \ll & K^{\varepsilon} \left( \frac{\delta K}{q} + \frac{E(K,q)}{\delta^3} \right) \textrm{.}
\end{eqnarray*}
The choice of some $\delta \leq 1$ such that $2$ is an integer power of $\zeta$ and satisfying
\begin{eqnarray*}
\delta & \asymp & \left( \frac{E(K,q) q}{K} \right)^{1/4} \textrm{,}
\end{eqnarray*}
is allowed since $q \ll K^{1/2}$ and yields
\begin{eqnarray}
\label{estimate D}
D(\mathcal{S};q,a) - D(\mathcal{S};q) & \ll & \frac{K^{1 - \vartheta + \varepsilon}}{q^{1 - 2 \vartheta + 1/600}} \textrm{.}
\end{eqnarray}
Averaging over $a$ coprime to $q$, we see that we can replace $D(\mathcal{S};q)$ in this estimate by
\begin{eqnarray*}
D^{\ast}(\mathcal{S};q) & = & \frac1{\varphi(q)} \# \left\{ (m_1,m_6,m_8) \in \mathcal{S} \cap \mathbb{Z}^3, \gcd(m_1m_6m_8, q) = 1 \right\} \textrm{.}
\end{eqnarray*}
Recalling the definition \eqref{definition S} of $\mathcal{S}$, we instantly get the bound
\begin{eqnarray*}
D^{\ast}(\mathcal{S};q) & \ll & \frac1{\varphi(q)} K_1 K_6 K_8 \textrm{.}
\end{eqnarray*}
Furthermore, assuming by symmetry that $K_1 \geq K_6, K_8$ and summing first over $m_1$ using the condition
$q K_9 < |m_1m_6m_8 - a| \leq 2 q K_9$, we also get
\begin{eqnarray*}
D^{\ast}(\mathcal{S};q) & \ll & \frac1{\varphi(q)} \left( q K_9 + K_6 K_8 \right) \\
& \ll & \frac1{\varphi(q)} \left( q K_9 + K^{2/3} \right) \textrm{.}
\end{eqnarray*}
Finally, we deduce
\begin{eqnarray*}
D^{\ast}(\mathcal{S};q) & \ll & \frac1{\varphi(q)} \min \left( K_1 K_6 K_8, q K_9 \right) + \frac{K^{2/3}}{\varphi(q)} \textrm{,}
\end{eqnarray*}
Putting together this bound and the estimate \eqref{estimate D}, we see that
\begin{eqnarray*}
D(\mathcal{S};q,a) & \ll & \frac1{\varphi(q)} \min \left( K_1 K_6 K_8, q K_9 \right) +
\frac{K^{1 - \vartheta + \varepsilon}}{q^{1 - 2 \vartheta + 1/600}} \textrm{.}
\end{eqnarray*}
Recalling that $q = m_{10}$ and summing over $m_3$, $m_5$, $m_7$ and $m_{10}$, we get
\begin{eqnarray*}
M & \ll & K_3 K_5 K_7 \left( \min(K_1 K_6 K_8, K_9 K_{10}) + K^{1 - \vartheta +\varepsilon} K_{10}^{2 \vartheta - 1/600} \right) \textrm{,}
\end{eqnarray*}
which completes the proof since $K_{10} \ll K_9$.
\end{proof}

We are now ready to prove lemma \ref{lemmalog}.

\begin{proof}
Let $Y_i \geq 1/2$ for $i = 1, \dots, 10 $ and $\mathcal{N} = \mathcal{N}(Y_1, \dots, Y_{10})$ be the contribution coming from the $(\eta_1, \dots, \eta_{10}) \in \mathcal{T}(B)$ subject to the conditions $Y_i < |\eta_i| \leq 2 Y_i$ for $i = 1, \dots, 10$. The height conditions imply that the following inequalities hold (otherwise $\mathcal{N} = 0$ and there is nothing to prove),
\begin{eqnarray}
\label{heighta}
Y_1 Y_4^2 Y_5 Y_6^2 Y_9 & \leq & B \textrm{,} \\
\label{heightb}
Y_2 Y_3^2 Y_4 Y_5^2 Y_7^3 & \leq & B \textrm{,} \\
\label{heightc}
Y_1 Y_2^2 Y_3 Y_8^2 Y_{10} & \leq & B \textrm{,} \\
\label{heightd}
Y_1 Y_2 Y_3 Y_4 Y_5 Y_6 Y_7 Y_8 & \leq & B \textrm{.}
\end{eqnarray}
Define
\begin{eqnarray*}
K & = & \frac{B}{Y_2 Y_3 Y_4 Y_5 Y_7} \textrm{.}
\end{eqnarray*}
The conditions \eqref{heightb} and \eqref{heightd} can be rewritten as $Y_3 Y_5 Y_7^2 \leq K$ and $Y_1 Y_6 Y_8 \leq K$. We can therefore apply lemma \ref{lemma bound}. By summing over $\eta_2$ and $\eta_4$, we get
\begin{eqnarray*}
\mathcal{N} & \ll & Y_2 Y_3 Y_4 Y_5 Y_7 \left( \min(Y_1 Y_6 Y_8, Y_9 Y_{10}) +
K^{1 - \vartheta + \varepsilon} (Y_9 Y_{10})^{\vartheta - 1/1200} \right) \textrm{.}
\end{eqnarray*}
Let respectively $\mathcal{N}'$ and $\mathcal{N}''$ denote the first and the second terms of the right-hand side. In the following estimations, the notation $\sum_{\widehat{Y}}$ indicates that the summation is over all the $Y_i \neq Y$. Let us first estimate the contribution of $\mathcal{N}''$. For this, we first sum over $Y_{10}$ using the condition $Y_9 Y_{10} \ll K$, and we choose $\varepsilon = 1/2400$, we get
\begin{eqnarray*}
\sum_{Y_i} \mathcal{N}'' & \ll & \sum_{\widehat{Y_{10}}} Y_2 Y_3 Y_4 Y_5 Y_7 K^{1 - 1/2400} \\
& = & B^{1-1/1200} \sum_{\widehat{Y_{10}}} (Y_2 Y_3 Y_4 Y_5 Y_7)^{1/2400} \\
& \ll & B \sum_{\widehat{Y_2}, \widehat{Y_{10}}} Y_3^{-1/2400} Y_5^{-1/2400} Y_7^{-1/1200} \\
& \ll & B  \log(B)^5 \textrm{,}
\end{eqnarray*}
where we have summed over $Y_2$ using the condition \eqref{heightb}. Let us now turn to the estimation of the contribution of $\mathcal{N}'$. We notice that we have
\begin{eqnarray*}
\mathcal{N}' & \leq & Y_2 Y_3 Y_4 Y_5 Y_7 (Y_1 Y_6 Y_8)^{3/4} (Y_9 Y_{10})^{1/4} \textrm{.}
\end{eqnarray*}
It turns out that the investigation of the contribution of this quantity is sufficient for our purpose. We successively sum over $Y_9$, $Y_{10}$, $Y_8$ and $Y_7$ using respectively the condtions \eqref{heighta}, \eqref{heightc}, \eqref{heightd} and \eqref{heightb}. We obtain
\begin{eqnarray*}
\sum_{Y_i} \mathcal{N}' & \ll & B^{1/2} \sum_{\widehat{Y_9}, \widehat{Y_{10}}}
Y_2^{1/2} Y_3^{3/4} Y_4^{1/2} Y_5^{3/4} Y_7 Y_1^{1/4} Y_6^{1/4} Y_8^{1/4} \\
& \ll & B^{3/4} \sum_{\widehat{Y_8}, \widehat{Y_9}, \widehat{Y_{10}}} Y_2^{1/4} Y_3^{1/2} Y_4^{1/4} Y_5^{1/2} Y_7^{3/4} \\
& \ll & B \sum_{\widehat{Y_7}, \widehat{Y_8}, \widehat{Y_9}, \widehat{Y_{10}}} 1 \textrm{.}
\end{eqnarray*}
This proves that if one of the variables among $\eta_1, \dots, \eta_6$, say $\eta_i$, is subject to the condition
$|\eta_i| \leq \log(B)^A$ for some fixed constant $A > 0$ then the overall contribution of $\mathcal{N}'$ is
$\ll_A B \log(B)^5 \log(\log(B))$. Furthermore, we could have summed over $\eta_1$ or $\eta_6$ instead of $\eta_8$ and over $\eta_3$ or $\eta_5$ instead of $\eta_7$ so the conclusion is also valid for $\eta_7$ and $\eta_8$, which completes the proof of lemma \ref{lemma bound}.
\end{proof}

\subsection{Setting up}

To be able to apply lemma \ref{lemma tau}, we need to make sure that in our setting the condition
$q \leq X^{1/2 + 1/230}$ holds. For this, we make the following assumption
\begin{eqnarray*}
|\eta_{10}| & \leq & |\eta_9| \textrm{.}
\end{eqnarray*}
Thanks to the equation \eqref{torsor} and the height conditions \eqref{height2} and \eqref{height4}, we get the new condition
\begin{eqnarray}
\label{new}
\eta_{10}^2 & \leq & 2 \frac{B}{\eta_2\eta_3\eta_4\eta_5\eta_7} \textrm{.}
\end{eqnarray}
The symmetry between $\eta_9$ and $\eta_{10}$ revealed by $(\eta_4,\eta_5,\eta_6,\eta_9) \mapsto (\eta_2,\eta_3,\eta_8,\eta_{10})$ and the following lemma prove that we simply need to multiply our future main term by $2$ to take into account this new assumption.

\begin{lemma}
\label{equal}
Let $N_0(B)$ be the total number of $(\eta_1, \dots, \eta_{10}) \in \mathcal{T}(B)$ such that $|\eta_9| = |\eta_{10}|$. We have the upper bound
\begin{eqnarray*}
N_0(B) & \ll & B \log(B)^4 \textrm{.}
\end{eqnarray*}
\end{lemma}

\begin{proof}
We split the proof in three cases depending on which variable among $\eta_1$, $\eta_6$ and $\eta_8$ has greater absolute value (recall that $(\eta_1, \eta_ 6) \in \mathbb{Z}_{>0}^2$). We only treat the case where $\eta_1 \geq \eta_6, |\eta_8|$ since the two others are strictly identical. Note that the condition \eqref{new} is at our disposal here too. Let $N_0'$ be the number of $\eta_1$ and $\eta_{10}$ to be counted. We have
\begin{eqnarray*}
N_0' & \ll & \# \left\{ (\eta_1, \eta_{10}) \in \mathbb{Z}_{>0} \times \mathbb{Z}_{\neq 0}, 
\begin{array}{l}
\eta_{10}^2 = \pm \eta_3 \eta_5 \eta_7^2 \pm \eta_1 \eta_6 \eta_8 \\
\eqref{new}
\end{array}
 \right\}  \\
& \ll & \# \left\{ \eta_{10} \in \mathbb{Z}_{\neq 0}, 
\begin{array}{l}
\eta_{10}^2 \equiv \pm \eta_3 \eta_5 \eta_7^2 \imod{\eta_6 \eta_8} \\
\eqref{new}
\end{array}
 \right\} \\
& \ll & 2^{\omega(\eta_6 \eta_8)}
\left( \frac{B^{1/2}}{\eta_2^{1/2}\eta_3^{1/2}\eta_4^{1/2}\eta_5^{1/2}\eta_7^{1/2} \eta_6 |\eta_8| } + 1 \right) \textrm{,}
\end{eqnarray*}
where we have used $\gcd(\eta_3 \eta_5 \eta_7, \eta_6 \eta_8) = 1$ and the fact that for $a$ and $q$ two coprime integers and for $X \geq 1$, we have
\begin{eqnarray*}
\# \left\{ n \in \mathbb{Z}_{\neq 0},
\begin{array}{l}
n^2 \equiv a \imod{q} \\
|n| \leq X
\end{array}
\right \} & \ll & 2^{\omega(q)} \left( \frac{X}{q} + 1 \right) \textrm{.}
\end{eqnarray*}
Since we are dealing with the case where $\eta_1 \geq \eta_6, |\eta_8|$, we have
$\eta_1 \geq \eta_6^{1/2} |\eta_8|^{1/2}$. Combining this inequality with the condition \eqref{height4}, we get
\begin{eqnarray*}
\eta_2 \eta_3 \eta_4 \eta_5 \eta_6^{3/2} \eta_7 |\eta_8|^{3/2} & \leq & B \textrm{.}
\end{eqnarray*}
Summing our upper bound for $N_0'$ over $\eta_2$ for instance using this condition gives
\begin{eqnarray*}
\sum_{\eta_2} N_0' & \ll & 2^{\omega(\eta_6 \eta_8)}
\left( \frac{B}{\eta_3 \eta_4 \eta_5 \eta_7 \eta_6^{7/4} |\eta_8|^{7/4}}
+ \frac{B}{\eta_3 \eta_4 \eta_5 \eta_7 \eta_6^{3/2} |\eta_8|^{3/2} } \right) \textrm{.}
\end{eqnarray*}
Summing over the six remaining variables completes the proof of the lemma.
\end{proof}

Since $(\eta_9,\eta_{10}) \mapsto (-\eta_9,-\eta_{10})$ is a bijection on $\mathcal{T}(B)$, we can assume that $\eta_{10} > 0$ multiplying our main term by $2$ once again. Let $A > 0$ be a constant to be chosen later and let $N_A(B)$ be the overall contribution of the $(\eta_1, \dots, \eta_{10}) \in \mathcal{T}(B)$ subject to the conditions
\begin{eqnarray}
\label{sym}
& & 0 < \eta_{10} \leq |\eta_9| \textrm{,} \\
\label{log6}
& & \log(B)^A \leq \eta_6 \textrm{,} \\
\label{log8}
& & \log(B)^A \leq |\eta_8| \textrm{.}
\end{eqnarray}
Note that, combining the condition \eqref{height4} and the assumption \eqref{log8}, we get
\begin{eqnarray}
\label{control'}
\log(B)^A \eta_1 \eta_2 \eta_3 \eta_4 \eta_5 \eta_6 \eta_7 & \leq & B \textrm{.}
\end{eqnarray}
Lemmas \ref{T}, \ref{lemmalog} and \ref{equal} give us the following result.

\begin{lemma}
\label{N_A(B)}
For any fixed $A > 0$, we have the estimate
\begin{eqnarray*}
N_{U,H}(B) & = & 4 N_A(B) + O \left( B \log(B)^5 \log(\log(B)) \right) \textrm{.}
\end{eqnarray*}
\end{lemma}

The end of the proof is devoted to the estimation of $N_A(B)$.

\subsection{Application of lemma \ref{lemma tau}}

We take care of the equation \eqref{torsor} seeing it as a congruence modulo $\eta_{10}$.  For this, we replace $\eta_9$ by its value given by the equation \eqref{torsor} in the height conditions \eqref{height1} and \eqref{sym}. These conditions become
\begin{eqnarray*}
\eta_1 \eta_4^2 \eta_5 \eta_6^2 \left| \eta_3 \eta_5 \eta_7^2 + \eta_1 \eta_6 \eta_8 \right| & \leq & B \eta_{10} \textrm{,} \\
\eta_{10}^2 & \leq & \left| \eta_3 \eta_5 \eta_7^2 + \eta_1 \eta_6 \eta_8 \right| \textrm{,}
\end{eqnarray*}
and we carry on denoting them respectively by \eqref{height1} and \eqref{sym}. From now on, we set
$\boldsymbol{\eta} = (\eta_2, \eta_3, \eta_4, \eta_5, \eta_7, \eta_{10})$ and we consider that
$\boldsymbol{\eta} \in \mathbb{Z}_{> 0}^6$ is fixed and is subject to the conditions \eqref{height2} and \eqref{new} and to the coprimality conditions \eqref{gcd5}, \eqref{gcd6} and \eqref{gcd7}. Let $N(\boldsymbol{\eta},B)$ be the number of $(\eta_1, \eta_6, \eta_8, \eta_9) \in \mathbb{Z}_{> 0}^2 \times \mathbb{Z}_{\neq 0}^2$ satisfying the equation \eqref{torsor}, the conditions \eqref{height1}, \eqref{height3}, \eqref{height4}, \eqref{sym}, \eqref{log6}, \eqref{log8} and finally the coprimality conditions \eqref{gcd1}, \eqref{gcd2}, \eqref{gcd3} and \eqref{gcd4}. Recall the definition \eqref{varphi} of $\varphi^{\ast}$. The goal of this section is to prove the following lemma.

\begin{lemma}
\label{lemma inter}
For any fixed $A \geq 6$, we have the estimate
\begin{eqnarray*}
N(\boldsymbol{\eta},B) & = & \frac1{\eta_{10}}
\sum_{\substack{k_9|\eta_2 \\ \gcd(k_9,\eta_3) = 1}} \frac{\mu(k_9)}{k_9 \varphi^{\ast}(k_9\eta_{10})}
\sum_{\substack{k_1|\eta_2\eta_3\eta_4\eta_5\eta_7 \\ \gcd(k_1, k_9\eta_{10})=1}} \mu(k_1) \\
& & \sum_{\substack{k_6|\eta_2\eta_3\eta_5\eta_7 \\ \gcd(k_6, k_9\eta_{10})=1}} \mu(k_6)
\sum_{\ell_1 | k_9 \eta_{10}} \mu(\ell_1) C(\boldsymbol{\eta},B) + R(\boldsymbol{\eta},B) \textrm{,}
\end{eqnarray*}
where, with the notations $\eta_1 = k_1 \ell_1 \eta_1''$ and $\eta_6 = k_6 \eta_6'$,
\begin{eqnarray*}
C(\boldsymbol{\eta},B) & = & \# \left\{ \left( \eta_1'', \eta_6', \eta_8 \right) \in \mathbb{Z}_{> 0}^2 \times \mathbb{Z}_{\neq 0},
\begin{array}{l}
\gcd(\eta_6'\eta_8,k_9 \eta_{10}) = 1 \\
\eqref{height1}, \eqref{height3}, \eqref{height4} \\
\eqref{sym}, \eqref{log6}, \eqref{log8} \\
\eqref{gcd4} \\
\end{array}
\right\} \textrm{,}
\end{eqnarray*}
and where
\begin{eqnarray*}
\sum_{\boldsymbol{\eta}} R(\boldsymbol{\eta},B) & \ll & B \log(B)^5 \textrm{.}
\end{eqnarray*}
\end{lemma}

The achievement of lemma \ref{lemma inter} is that the summations over $\eta_1$, $\eta_6$ and $\eta_8$ have been carried out. That is why the torsor equation $\eqref{torsor}$ does not appear in the definition of 
$C(\boldsymbol{\eta},B)$ and also why we find $1/\eta_{10}$ in the main term of $N(\boldsymbol{\eta},B)$.

Let us remove the coprimality condition \eqref{gcd1} using a Möbius inversion. We get
\begin{eqnarray*}
N(\boldsymbol{\eta},B) & = & \sum_{k_9 | \eta_1\eta_2\eta_3\eta_5\eta_6\eta_7\eta_8} \mu(k_9) S_{k_9}(\boldsymbol{\eta},B) \textrm{,}
\end{eqnarray*}
where
\begin{eqnarray*}
S_{k_9}(\boldsymbol{\eta},B) & = & \# \left\{ \left( \eta_1, \eta_6, \eta_8, \eta_9' \right) \in \mathbb{Z}_{> 0}^2 \times \mathbb{Z}_{\neq 0}^2,
\begin{array}{l}
\eta_1 \eta_6 \eta_8 = -  \eta_3 \eta_5 \eta_7^2 - k_9 \eta_9' \eta_{10} \\
\eqref{height1}, \eqref{height3}, \eqref{height4} \\
\eqref{sym}, \eqref{log6}, \eqref{log8} \\
\eqref{gcd2}, \eqref{gcd3}, \eqref{gcd4} \\
\end{array}
\right\} \textrm{.}
\end{eqnarray*}
Since $\gcd(\eta_1 \eta_6 \eta_8, \eta_3 \eta_5 \eta_7) = 1$, $S_{k_9}(\boldsymbol{\eta},B)$ vanishes if $k_9$ and $\eta_1\eta_3\eta_5\eta_6\eta_7\eta_8$ are not coprime thus we can assume that
$\gcd(k_9, \eta_1\eta_3\eta_5\eta_6\eta_7\eta_8) = 1$. In addition, replacing the equation
$\eta_1 \eta_6 \eta_8 = -  \eta_3 \eta_5 \eta_7^2 - k_9 \eta_9' \eta_{10}$ by
$\eta_1 \eta_6 \eta_8 \equiv -  \eta_3 \eta_5 \eta_7^2 \imod{k_9 \eta_{10}}$ in $S_{k_9}(\boldsymbol{\eta},B)$ yields an error term $R_0(\boldsymbol{\eta},B)$ corresponding to the fact that $\eta_9'$ is not allowed to be equal to $0$. Otherwise, since
$\eta_1 \eta_6 \eta_8$ and $\eta_3 \eta_5 \eta_7$ are coprime, we would necessarily have
$\eta_1 = \eta_6 = |\eta_8| = \eta_3 = \eta_5 = \eta_7 = 1$ and thus the overall contribution of this error term is
\begin{eqnarray*}
\sum_{k_9, \boldsymbol{\eta}} |\mu(k_9)| R_0(\boldsymbol{\eta},B) & \ll & \sum_{\eta_2, \eta_4, \eta_{10}} 2^{\omega(\eta_2)} \\
& \ll & \sum_{\eta_2, \eta_4} 2^{\omega(\eta_2)} \frac{B^{1/2}}{\eta_2^{1/2}\eta_4^{1/2}} \\
& \ll & B \log(B)^2 \textrm{,}
\end{eqnarray*}
where we have summed over $\eta_{10}$ using the condition \eqref{new}. We now remove the two coprimality conditions \eqref{gcd2} and \eqref{gcd3} using Möbius inversions. We find that the main term of $N(\boldsymbol{\eta},B)$ is equal to
\begin{eqnarray*}
\sum_{\substack{k_9|\eta_2 \\ \gcd(k_9,\eta_3\eta_5\eta_7) = 1}} \mu(k_9)
\sum_{k_1|\eta_2\eta_3\eta_4\eta_5\eta_7\eta_{10}} \mu(k_1) \sum_{k_6|\eta_2\eta_3\eta_5\eta_7\eta_{10}} \mu(k_6) S(\boldsymbol{\eta},B) \textrm{,}
\end{eqnarray*}
where
\begin{eqnarray*}
S(\boldsymbol{\eta},B) & = & \# \left\{ \left( \eta_1', \eta_6', \eta_8 \right) \in \mathbb{Z}_{> 0}^2 \times \mathbb{Z}_{\neq 0},
\begin{array}{l}
k_1 k_6 \eta_1' \eta_6' \eta_8 \equiv -  \eta_3 \eta_5 \eta_7^2 \imod{k_9 \eta_{10}} \\
\eqref{height1}, \eqref{height3}, \eqref{height4} \\
\eqref{sym}, \eqref{log6}, \eqref{log8} \\
\eqref{gcd4} \\
\end{array}
\right\} \textrm{,}
\end{eqnarray*}
and where we use the notations $\eta_1 = k_1 \eta_1'$ and $\eta_6 = k_6 \eta_6'$. Since $\eta_3 \eta_5 \eta_7$ and $k_9 \eta_{10}$ are coprime, we can add the condition $\gcd(k_1 k_6, k_9 \eta_{10}) = 1$ in the summations over $k_1$ and $k_6$ and the congruence can therefore be rewritten as $\eta_1' \eta_6' \eta_8 \equiv a \imod{k_9 \eta_{10}}$ where we have set
$a = - \left( k_1 k_6 \right)^{-1} \eta_3 \eta_5 \eta_7^2$. Define
\begin{eqnarray}
\label{X}
X & = & \frac{B}{k_1 k_6 \eta_2 \eta_3 \eta_4 \eta_5 \eta_7} \textrm{.}
\end{eqnarray}
We see that we are almost in position to apply lemma \ref{lemma tau}, however, we still need to check that we can assume that
$k_9 \eta_{10} \leq X^{1/2 + 1/230}$. The condition \eqref{new} can be rewritten as
\begin{eqnarray*}
\eta_{10} & \leq & \left( 2 k_1 k_6 \right)^{1/2} X^{1/2} \textrm{,}
\end{eqnarray*}
thus we need to check that the summation over $k_9$ can be restricted to 
\begin{eqnarray*}
k_9 & \leq & \left( 2 k_1 k_6 \right)^{-1/2} X^{1/230} \textrm{.}
\end{eqnarray*}
Indeed, let $N_1(\boldsymbol{\eta},B)$ be the contribution to $N(\boldsymbol{\eta},B)$ under the assumption
\begin{eqnarray*}
k_9 & > & \left( 2 k_1 k_6 \right)^{-1/2} X^{1/230} \textrm{.}
\end{eqnarray*}
We clearly have
\begin{eqnarray*}
S(\boldsymbol{\eta},B) & \leq & \# \left\{ \left( \eta_1', \eta_6', \eta_8 \right) \in \mathbb{Z}_{> 0}^2 \times \mathbb{Z}_{\neq 0},
\begin{array}{l}
 \eta_1' \eta_6' \eta_8 \equiv a \imod{k_9 \eta_{10}} \\
\eta_1' \eta_6' |\eta_8| \leq X \\
\end{array}
\right\} \\
& = & \sum_{\substack{1 \leq |n| \leq X \\ n \equiv a \imod{k_9 \eta_{10}}}} \tau_3(|n|) \\
& \ll & X^{\varepsilon} \left( \frac{X}{k_9 \eta_{10}} + 1 \right) \textrm{,}
\end{eqnarray*}
for any $\varepsilon > 0$. Let us use $k_9^{1/2} > \left( 2 k_1 k_6 \right)^{-1/4} X^{1/460}$, we obtain
\begin{eqnarray}
\label{assumption k_9}
S(\boldsymbol{\eta},B) & \ll & \left( k_1 k_6 \right)^{-1/4}
\frac{X^{1 - 1/460 + \varepsilon}}{k_9^{1/2} \eta_{10}} + X^{\varepsilon} \textrm{.}
\end{eqnarray}
Note that when one of the $k_i$ appears at the denominator of an error term then the arithmetic function involved by the Möbius inversion has average order $O(1)$ and consequently does not play any part in the estimation of the overall contribution of this error term. We thereby obtain
\begin{eqnarray*}
N_1(\boldsymbol{\eta},B) & \ll & \frac1{\eta_{10}}
\left( \frac{B}{\eta_2 \eta_3 \eta_4 \eta_5 \eta_7} \right)^{1 - 1/460 + \varepsilon} +
2^{\omega(\eta_2)} \left( \frac{B}{\eta_2 \eta_3 \eta_4 \eta_5 \eta_7} \right)^{\varepsilon} \textrm{.}
\end{eqnarray*}
Let us sum over $\eta_{10}$ using \eqref{new}, we get 
\begin{eqnarray*}
\sum_{\eta_{10}} N_1(\boldsymbol{\eta},B) & \ll & 
\left( \frac{B}{\eta_2 \eta_3 \eta_4 \eta_5 \eta_7} \right)^{1 - 1/460 + 2 \varepsilon} +
2^{\omega(\eta_2)} \left( \frac{B}{\eta_2 \eta_3 \eta_4 \eta_5 \eta_7} \right)^{1/2 + \varepsilon} \textrm{.}
\end{eqnarray*}
Choosing $\varepsilon = 1/1840$ and summing over $\eta_4$ using \eqref{height2}, we finally obtain that the overall contribution of the first term is
\begin{eqnarray*}
& \ll & \sum_{\eta_2, \eta_3, \eta_5, \eta_7} \frac{B}{\eta_2 \eta_3^{1+1/920} \eta_5^{1+1/920} \eta_7^{1+1/460}} \\
& \ll & B \log(B) \textrm{,}
\end{eqnarray*}
and the overall contribution of the second term is
\begin{eqnarray*}
& \ll & \sum_{\eta_2, \eta_3, \eta_5, \eta_7} 2^{\omega(\eta_2)}
\frac{B}{\eta_2 \eta_3^{3/2 - 1/1840} \eta_5^{3/2 - 1/1840} \eta_7^{2 - 1/920}} \\
& \ll & B \log(B)^2 \textrm{.}
\end{eqnarray*}
Let us sum up what we have done until now. We have proved that
\begin{eqnarray*}
N(\boldsymbol{\eta},B) & = &  \sum_{\substack{k_1|\eta_2\eta_3\eta_4\eta_5\eta_7 \\ \gcd(k_1,\eta_{10}) = 1}} \mu(k_1)
\sum_{\substack{k_6|\eta_2\eta_3\eta_5\eta_7 \\ \gcd(k_6,\eta_{10}) = 1}} \mu(k_6)
\! \! \! \! \sum_{\substack{k_9|\eta_2 \\ k_9 \leq \left( 2 k_1 k_6 \right)^{-1/2} X^{1/230} \\
\gcd(k_9, k_1 k_6 \eta_3\eta_5\eta_7) = 1}} \! \! \! \! \mu(k_9)  S(\boldsymbol{\eta},B) \\
& & + R_1(\boldsymbol{\eta},B)\textrm{,}
\end{eqnarray*}
where $X$ is defined in \eqref{X} and where
\begin{eqnarray*}
\sum_{\boldsymbol{\eta}} R_1(\boldsymbol{\eta},B) & \ll & B \log(B)^2 \textrm{.}
\end{eqnarray*}
We now aim to apply lemma \ref{lemma tau} with $b = k_6 \eta_3 \eta_4 \eta_5 \eta_7 \eta_{10}$ and
\begin{eqnarray*}
X_1 & = & \frac{B \eta_{10}}{k_1^2 k_6^3 \eta_4^2 \eta_5} \textrm{,} \\
X_2 & = & \frac{B}{k_1 \eta_2^2 \eta_3 \eta_{10}} \textrm{,} \\
T & = & \frac{\eta_3 \eta_5 \eta_7^2}{k_1 k_6} \textrm{,} \\
Z & = & \frac{\eta_{10}^2}{k_1 k_6} \textrm{,}
\end{eqnarray*}
and finally $L_1 = \log(B)^A / k_6$ and $L_2 = \log(B)^A$. The condition \eqref{height2} shows that we actually have $T \leq X$ and furthermore, we also have $k_9 \eta_{10} \leq X^{1/2 + 1/230}$ thus we can apply lemma \ref{lemma tau}. Recall the definitions of $\vartheta$, $\lambda$ and $\sigma_{- \lambda}$ respectively given in lemmas~\ref{pre lemma tau} and \ref{Fri-Iwa lemma 2} and in \eqref{sigma}. We get that for any fixed $\varepsilon > 0$,
\begin{eqnarray*}
S(\boldsymbol{\eta},B) - S^{\ast}(\boldsymbol{\eta},B) & \ll & 
\sigma_{- \lambda}(b)^{1/4} \frac{X^{1 - \vartheta + \varepsilon}}{(k_9 \eta_{10})^{1 - 2 \vartheta + 1/600}} +
\frac{X \log(X)}{\varphi(k_9 \eta_{10})} \frac{k_6}{\log(B)^A} \textrm{,}
\end{eqnarray*}
where
\begin{eqnarray*}
S^{\ast}(\boldsymbol{\eta},B) & = & \frac1{\varphi(k_9 \eta_{10})}
\# \left\{ \left( \eta_1', \eta_6', \eta_8 \right) \in \mathbb{Z}_{> 0}^2 \times \mathbb{Z}_{\neq 0},
\begin{array}{l}
\gcd(\eta_1'\eta_6'\eta_8,k_9 \eta_{10}) = 1 \\
\eqref{height1}, \eqref{height3}, \eqref{height4} \\
\eqref{sym}, \eqref{log6}, \eqref{log8} \\
\eqref{gcd4} \\
\end{array}
\right\} \textrm{.}
\end{eqnarray*}
Let us estimate the contribution of these two error terms. As explained earlier, the Möbius inversions do not intervene in the estimation of the first error term and neither does $\sigma_{- \lambda}(b)^{1/4}$ since its average order is $O(1)$. Summing over $\eta_{10}$ using the condition \eqref{new}, we easily get that the overall contribution of the first error term is
\begin{eqnarray*}
\sum_{\boldsymbol{\eta}} 
\frac{B^{1 - \vartheta + \varepsilon}}{(\eta_2\eta_3\eta_4\eta_5\eta_7)^{1 - \vartheta + \varepsilon}
\eta_{10}^{1 - 2 \vartheta + 1/600}} & \ll & \sum_{\eta_2, \eta_3, \eta_4, \eta_5, \eta_7}
\frac{B^{1 - 1/1200 + \varepsilon}}{(\eta_2\eta_3\eta_4\eta_5\eta_7)^{1 - 1/1200 + \varepsilon}} \textrm{.}
\end{eqnarray*}
Choosing $\varepsilon = 1/2400$ and summing over $\eta_2$ using the height condition \eqref{height2}, we deduce
\begin{eqnarray*}
\sum_{\eta_2, \eta_3, \eta_4, \eta_5, \eta_7}
\frac{B^{1 - 1/2400}}{(\eta_2\eta_3\eta_4\eta_5\eta_7)^{1 - 1/2400}} & \ll &
\sum_{\eta_3, \eta_4, \eta_5, \eta_7} \frac{B}{\eta_3^{1 + 1/2400} \eta_4 \eta_5^{1+1/2400} \eta_7^{1+1/1200}} \\
& \ll & B \log(B) \textrm{.}
\end{eqnarray*}
In addition, the overall contribution of the second error term is bounded by
\begin{eqnarray*}
\sum_{\boldsymbol{\eta}} 2^{\omega(\eta_2 \eta_3 \eta_5 \eta_7)} \frac{B \log(B)^{1-A}}{\eta_2 \eta_3 \eta_4 \eta_5 \eta_7 \eta_{10}} & \ll & B \log(B)^{11-A} \textrm{,}
\end{eqnarray*}
which is satisfactory if $A \geq 6$. Let us now prove that we can remove the condition
$k_9 \leq \left( 2 k_1 k_6 \right)^{-1/2} X^{1/230}$ from the sum over $k_9$. We clearly have 
\begin{eqnarray*}
S^{\ast}(\boldsymbol{\eta},B) & \ll & \frac{X^{1 + \varepsilon}}{\varphi(k_9 \eta_{10})} \textrm{.}
\end{eqnarray*}
Therefore, mimicking what we have done to deal with the first error term in \eqref{assumption k_9} proves that the contribution corresponding to $k_9 > \left( 2 k_1 k_6 \right)^{-1/2} X^{1/230}$ is $\ll B \log(B)$ and thus the condition
$k_9 \leq \left( 2 k_1 k_6 \right)^{-1/2} X^{1/230}$ can actually be removed from the sum over $k_9$. We also see that since $\gcd(\eta_2,\eta_5\eta_7) = 1$, we can remove the condition $\gcd(k_9,\eta_5\eta_7) = 1$ from the sum over $k_9$. To complete the proof of lemma \ref{lemma inter}, we simply notice that, with the notation $\eta_1 = k_1 \ell_1 \eta_1''$, we have
\begin{eqnarray*}
S^{\ast}(\boldsymbol{\eta},B) & = & \sum_{\ell_1 | k_9 \eta_{10}} \mu(\ell_1) C(\boldsymbol{\eta},B) \textrm{.}
\end{eqnarray*}

\subsection{Summations over $\eta_8$, $\eta_6'$ and $\eta_7$}

\label{summations section}

In the estimation of the main term obtained in lemma \ref{lemma inter}, we can choose the order in which we want to sum our variables at our best convenience. We decide to start by summing over $\eta_8$, $\eta_6'$ and $\eta_7$. We define
$\boldsymbol{\eta}' = (\eta_1'',\eta_2, \eta_3, \eta_4, \eta_5, \eta_{10})$ and we introduce the notation
\begin{eqnarray*}
\boldsymbol{\eta}'^{(r_1,r_2,r_3,r_4,r_5,r_{10})} & = & \eta_1''^{r_1} \eta_2^{r_2} \eta_3^{r_3} \eta_4^{r_4} \eta_5^{r_5} \eta_{10}^{r_{10}} \textrm{,}
\end{eqnarray*}
for $(r_1,r_2,r_3,r_4,r_5,r_{10}) \in \mathbb{Q}^6$. To ease the understanding of the height conditions we introduce the following quantities
\begin{align*}
& Y_8 = \frac{B^{1/2}}{\boldsymbol{\eta}'^{(1/2,1,1/2,0,0,1/2)}} \textrm{,} &
& Y_8' = \frac{Y_8}{k_1^{1/2} \ell_1^{1/2}} \textrm{,} \\
& Y_6 = \frac{B^{1/6}}{\boldsymbol{\eta}'^{(1/2,-1/3,-1/6,2/3,1/3,-1/2)}} \textrm{,} &
& Y_6' = \frac{Y_6}{k_1^{1/2} \ell_1^{1/2} k_6} \textrm{,} \\
& Y_7 = \frac{B^{1/3}}{\boldsymbol{\eta}'^{(0,1/3,2/3,1/3,2/3,0)}} \textrm{,} &
\end{align*}
and recalling the definition \eqref{equation h} of the function $h$, it is easy to check that the height conditions \eqref{height1}, \eqref{height2}, \eqref{height3} and \eqref{height4} can be summed up as
\begin{eqnarray*}
h \left( \frac{\eta_8}{Y_8'}, \frac{\eta_6'}{Y_6'}, \frac{\eta_7}{Y_7} \right) & \leq & 1 \textrm{.}
\end{eqnarray*}
Set $\mathcal{L} = k_1^{1/2} \ell_1^{1/2} \log(B)^A$. We also define the real-valued functions
\begin{eqnarray*}
g_1 & : & (t_6,t_7,t;\boldsymbol{\eta}',B) \mapsto
\int_{h(t_8,t_6,t_7) \leq 1, t \leq \left| t_6 t_8 + t_7^2 \right|, |t_8|Y_8 \geq \mathcal{L}} \D t_8 \textrm{,} \\
g_2 & : & (t_7,t;\boldsymbol{\eta}',B) \mapsto \int_{t_6 Y_6 \geq \mathcal{L}} g_1(t_6,t_7,t;\boldsymbol{\eta}',B) \D t_6 \textrm{,} \\
g_3 & : & (t_7,t) \mapsto \int \int_{t_6> 0, h(t_8,t_6,t_7) \leq 1, t \leq \left| t_6 t_8 + t_7^2 \right|} \D t_8 \D t_6 \textrm{,} \\
g_4 & : & t \mapsto \int \int \int_{t_6, t_7 > 0, h(t_8,t_6,t_7) \leq 1, t \leq \left| t_6 t_8 + t_7^2 \right|}
\D t_8 \D t_6 \D t_7 \textrm{.}
\end{eqnarray*}
The condition $t \leq \left| t_6 t_8 + t_7^2 \right|$ is here to take into account the condition \eqref{sym} which can be rewritten with our new notations as
\begin{eqnarray*}
\kappa & \leq & \left| \frac{\eta_6'}{Y_6'} \frac{\eta_8}{Y_8'} + \left( \frac{\eta_7}{Y_7} \right)^2 \right| \textrm{,}
\end{eqnarray*}
where $\kappa$ is defined by
\begin{eqnarray*}
\kappa & = & \frac{\eta_{10}^2}{\eta_3 \eta_5 Y_7^2} \textrm{.}
\end{eqnarray*}

\begin{lemma}
\label{bounds}
We have the bounds
\begin{eqnarray*}
g_1 \left(t_6,t_7,t;\boldsymbol{\eta}',B \right) & \ll & t_6^{-1} t_7^{-1} \textrm{,} \\
g_2 \left(t_7,t;\boldsymbol{\eta}',B \right) & \ll & 1 \textrm{.}
\end{eqnarray*}
\end{lemma}

\begin{proof}
The bound for $g_1$ follows from the inequality $t_6 t_7 |t_8| \leq 1$. In addition, the two conditions
$t_6^2 \left| t_7^2 + t_6t_8 \right| \leq 1$ and $|t_8| \leq 1$ imply
$g_1 \left(t_6,t_7,t;\boldsymbol{\eta}',B \right) \leq 2 \min \left( t_6^{-3}, 1 \right)$. Integrating this minimum over $t_6$ gives the bound for $g_2$.
\end{proof}

It is immediate to check that $\boldsymbol{\eta}'$ is restricted to lie in the region $\mathcal{V}$ defined by
\begin{eqnarray}
\label{V}
\mathcal{V} & = & \left\{ \boldsymbol{\eta}' \in \mathbb{Z}_{> 0}^6, Y_7 \geq 1, Y_8 \geq \log(B)^A \right\} \textrm{.}
\end{eqnarray} 
We consider from now on that $\boldsymbol{\eta}' \in \mathcal{V}$ and $\eta_7 \in \mathbb{Z}_{> 0}$ are fixed and are subject to the height condition \eqref{height2} and to the coprimality conditions \eqref{gcd5}, \eqref{gcd6} and \eqref{gcd7}. We set
\begin{eqnarray*}
P(\boldsymbol{\eta}',\eta_7,B) & = & \# \left\{(\eta_6', \eta_8) \in \mathbb{Z}_{> 0} \times \mathbb{Z}_{\neq 0}, (\eta_1'',\eta_6',\eta_8) \in C(\boldsymbol{\eta},B) \right\} \textrm{,}
\end{eqnarray*}
and
\begin{eqnarray}
\label{definition  N}
N(\boldsymbol{\eta}',\eta_7,B) & = & \frac1{\eta_{10}}
\sum_{\substack{k_9|\eta_2 \\ \gcd(k_9,\eta_3) = 1}} \frac{\mu(k_9)}{k_9 \varphi^{\ast}(k_9\eta_{10})}
\sum_{\substack{k_1|\eta_2\eta_3\eta_4\eta_5\eta_7 \\ \gcd(k_1, k_9\eta_{10})=1}} \mu(k_1) \\
\notag
& & \sum_{\substack{k_6|\eta_2\eta_3\eta_5\eta_7 \\ \gcd(k_6, k_9\eta_{10})=1}} \mu(k_6)
\sum_{\ell_1 | k_9 \eta_{10}} \mu(\ell_1) P(\boldsymbol{\eta}',\eta_7,B)  \textrm{.}
\end{eqnarray}
Using the estimates of lemmas \ref{N_A(B)} and \ref{lemma inter}, we see that for any fixed $A \geq 6$, we have
\begin{eqnarray}
\label{estimate N_{U,H}}
\ \ \ \ \ \ \ \ N_{U,H}(B) & = & 4 \! \! \sum_{\substack{\boldsymbol{\eta}' \in \mathcal{V} \\ \eqref{gcd6}, \eqref{gcd7}}}
\sum_{\substack{ \eta_7 \in \mathbb{Z}_{> 0} \\ \eqref{height2}, \eqref{gcd5}}} \! \! N(\boldsymbol{\eta}',\eta_7,B) 
+ O \left( B \log(B)^5 \log(\log(B)) \right) \textrm{.}
\end{eqnarray}
We now prove the following result.

\begin{lemma}
\label{summations 6, 8}
For any fixed $A \geq 10$, we have the estimate
\begin{eqnarray*}
N(\boldsymbol{\eta}',\eta_7,B) & = & \zeta(2)^{-1} \frac{Y_6Y_8}{\eta_{10}}
g_2 \left( \frac{\eta_7}{Y_7}, \kappa; \boldsymbol{\eta}', B \right)
\theta_1(\boldsymbol{\eta}') \theta_2(\boldsymbol{\eta}',\eta_7)
+ R(\boldsymbol{\eta}',\eta_7,B) \textrm{,}
\end{eqnarray*}
where $\theta_1(\boldsymbol{\eta}')$ and $\theta_2(\boldsymbol{\eta}',\eta_7)$ are arithmetic functions respectively defined in \eqref{theta1} and \eqref{theta2} and where
\begin{eqnarray*}
\sum_{\boldsymbol{\eta}', \eta_7} R(\boldsymbol{\eta}',\eta_7,B) & \ll & B \log(B)^5 \textrm{.}
\end{eqnarray*}
\end{lemma}

Note that it is clear that the two coprimality conditions remaining in $C(\boldsymbol{\eta},B)$ can be rewritten as
$\gcd(\eta_8,\eta_3 \eta_4 \eta_5 k_6 \eta_6' \eta_7 k_9 \eta_{10}) = 1$ and $\gcd(\eta_6', k_9 \eta_{10}) = 1$. Recall the condition \eqref{control'} which can be rewritten with our new conditions as
\begin{eqnarray}
\label{control}
\eta_6' & \leq & \frac{B \log(B)^{-A}}{k_1 \ell_1 k_6 \boldsymbol{\eta}'^{(1,1,1,1,1,0)} \eta_7} \textrm{.}
\end{eqnarray}
Using a Möbius inversion and the trivial estimate
\begin{eqnarray*}
\# \left\{ n \in \mathbb{Z} \cap [t_1,t_2] \right\} & = & t_2 - t_1 + O(1) \textrm{,}
\end{eqnarray*}
we obtain
\begin{eqnarray*}
P(\boldsymbol{\eta}',\eta_7,B) & = & \sum_{\substack{\eta_6' \in \mathbb{Z}_{>0}, \eqref{control} \\ \gcd(\eta_6', k_9 \eta_{10})= 1}} \left( \varphi^{\ast}(c \eta_6') Y_8'g_1 \left( \frac{\eta_6'}{Y_6'}, \frac{\eta_7}{Y_7}, \kappa; \boldsymbol{\eta}', B \right)
+ O \left( 2^{\omega(c \eta_6')} \right) \right) \textrm{,}
\end{eqnarray*}
where we have set $c = \eta_3 \eta_4 \eta_5 k_6 \eta_7 k_9 \eta_{10}$. Recalling the definition \eqref{definition  N} of $N(\boldsymbol{\eta}',\eta_7,B)$, we see that the overall contribution of the error term is
\begin{eqnarray*}
\sum_{\boldsymbol{\eta}', \eta_7} 2^{\omega(\eta_2 \eta_3 \eta_4 \eta_5 \eta_7 \eta_{10})}
\frac{B \log(B)^{1-A}}{\boldsymbol{\eta}'^{(1,1,1,1,1,1)} \eta_7} & \ll & B \log(B)^{14 - A} \textrm{,}
\end{eqnarray*}
which is satisfactory if $A \geq 9$. Writing
$\varphi^{\ast}(c \eta_6') = \varphi^{\ast}(c) \varphi^{\ast}(\eta_6') \varphi^{\ast}( \gcd(c, \eta_6'))^{-1} $, we see that we can use lemma \ref{arithmetic preliminary 1} to sum over $\eta_6'$. Recall the definition \eqref{varphi+} of $\varphi^{+}$. We obtain that for any fixed $A \geq 9$ and $0 < \gamma \leq 1$,
\begin{eqnarray*}
P(\boldsymbol{\eta}',\eta_7,B) & = & \zeta(2)^{-1} \varphi^{\ast}(k_9 \eta_{10})
\varphi^{+}(\eta_3 \eta_4 \eta_5 k_6 \eta_7 k_9 \eta_{10}) Y_6' Y_8'
g_2 \left( \frac{\eta_7}{Y_7}, \kappa; \boldsymbol{\eta}', B \right) \\
& & + O \left( Y_8'  2^{\omega(k_9 \eta_{10})} \log(B) \sup_{t_6 Y_6 \geq \mathcal{L}}
g_1 \left( t_6, \frac{\eta_7}{Y_7}, \kappa; \boldsymbol{\eta}', B \right) \right) \textrm{.}
\end{eqnarray*}
Using the bound for $g_1$ proved in lemma \ref{bounds}, we get
\begin{eqnarray*}
\sup_{t_6 Y_6 \geq \mathcal{L}} g_1 \left( t_6, \frac{\eta_7}{Y_7}, \kappa; \boldsymbol{\eta}', B \right) & \ll &
\frac{Y_6}{\mathcal{L}} \frac{Y_7}{\eta_7} \textrm{.}
\end{eqnarray*}
Noticing that
\begin{eqnarray*}
\frac{Y_6 Y_7 Y_8}{\eta_{10}} & = & \frac{B}{\boldsymbol{\eta}'^{(1,1,1,1,1,1)}} \textrm{,}
\end{eqnarray*}
we see that the overall contribution of this error term is
\begin{eqnarray*}
\sum_{\boldsymbol{\eta}', \eta_7} 2^{\omega(\eta_2 \eta_{10})} 2^{\omega(\eta_2 \eta_3 \eta_5 \eta_7)}
\frac{B \log(B)^{1-A}}{\boldsymbol{\eta}'^{(1,1,1,1,1,1)} \eta_7} & \ll & B \log(B)^{15 - A} \textrm{,}
\end{eqnarray*}
which is satisfactory if $A \geq 10$. Recalling the definition \eqref{definition N} of $N(\boldsymbol{\eta}',\eta_7,B)$, we see that for any fixed $A \geq 10$, we have the estimate
\begin{eqnarray*}
N(\boldsymbol{\eta}',\eta_7,B) & = & \zeta(2)^{-1} \frac{Y_6Y_8}{\eta_{10}}
g_2 \left( \frac{\eta_7}{Y_7}, \kappa; \boldsymbol{\eta}', B \right) \theta(\boldsymbol{\eta}',\eta_7)
+ R(\boldsymbol{\eta}',\eta_7,B) \textrm{,}
\end{eqnarray*}
where
\begin{eqnarray*}
\sum_{\boldsymbol{\eta}',\eta_7} R(\boldsymbol{\eta}',\eta_7,B) &  \ll & B \log(B)^5 \textrm{,}
\end{eqnarray*}
and where
\begin{eqnarray*}
\theta(\boldsymbol{\eta}',\eta_7) & = &
\sum_{\substack{k_9|\eta_2 \\ \gcd(k_9,\eta_3) = 1}} \frac{\mu(k_9)}{k_9 \varphi^{\ast}(k_9\eta_{10})}
\sum_{\substack{k_1|\eta_2\eta_3\eta_4\eta_5\eta_7 \\ \gcd(k_1, k_9\eta_{10})=1}} \frac{\mu(k_1)}{k_1} 
\sum_{\substack{k_6|\eta_2\eta_3\eta_5\eta_7 \\ \gcd(k_6, k_9\eta_{10})=1}} \frac{\mu(k_6)}{k_6} \\
& & \sum_{\ell_1 | k_9 \eta_{10}} \frac{\mu(\ell_1)}{\ell_1} \varphi^{\ast}(k_9 \eta_{10})
\varphi^{+}(\eta_3 \eta_4 \eta_5 k_6 \eta_7 k_9 \eta_{10}) \textrm{.}
\end{eqnarray*}
It is straightforward to check that, given $a, b \in \mathbb{Z}_{\geq 1}$, the following equality holds
\begin{eqnarray*}
\sum_{\substack{k|n \\ \gcd(k,a) = 1}} \frac{\mu(k)}{k} \varphi^{+}(kb) & = &
\frac{\varphi^{\ast}(\gcd(n,b))}{\varphi^{\ast}(\gcd(n,a,b))}
\frac{\varphi^{+}(n) \varphi^{+}(b)}{\varphi^{+}(\gcd(n,ab))} \textrm{.}
\end{eqnarray*}
Using this equality, it is easy to simplify the expression of $\theta(\boldsymbol{\eta}',\eta_7)$ and we obtain
\begin{eqnarray*}
\theta(\boldsymbol{\eta}',\eta_7) & = & \varphi^{\ast}(\eta_2\eta_3\eta_5\eta_7)
\varphi^{\ast}(\eta_2\eta_3\eta_4\eta_5\eta_7\eta_{10}) \varphi^{+}(\eta_2\eta_3\eta_4\eta_5\eta_7\eta_{10}) \textrm{.}
\end{eqnarray*}
Bearing in mind that our next step is to sum over $\eta_7$, we set
\begin{eqnarray*}
\theta(\boldsymbol{\eta}',\eta_7) & = & \theta_1(\boldsymbol{\eta}') \theta_2(\boldsymbol{\eta}',\eta_7) \textrm{,}
\end{eqnarray*}
where
\begin{eqnarray}
\label{theta1}
\theta_1(\boldsymbol{\eta}') & = & \varphi^{\ast}(\eta_2\eta_3\eta_5) \varphi^{\ast}(\eta_2\eta_3\eta_4\eta_5\eta_{10}) \varphi^{+}(\eta_2\eta_3\eta_4\eta_5\eta_{10}) \textrm{,}
\end{eqnarray}
and
\begin{eqnarray}
\label{theta2}
\theta_2(\boldsymbol{\eta}',\eta_7) & = & \frac{\varphi^{\ast}(\eta_7)^2}{\varphi^{\ast}(\gcd(\eta_7, \eta_3 \eta_5))^2}
\frac{\varphi^{+}(\eta_7)}{\varphi^{+}(\gcd(\eta_7, \eta_3 \eta_5))} \textrm{.}
\end{eqnarray}
This completes the proof of lemma \ref{summations 6, 8}.

The following handlings aim to remove the conditions $t_6 Y_6 \geq \mathcal{L}$ (more exactly replace it by $t_6 > 0$) and
$|t_8| Y_8 \geq \mathcal{L}$ from the integral defining $g_2$ in the main term of $N(\boldsymbol{\eta}',\eta_7,B)$. This will have for effect to replace $g_2 \left( \eta_7 / Y_7, \kappa; \boldsymbol{\eta}', B \right)$ by $g_3 \left( \eta_7 / Y_7, \kappa \right)$. For short, we set
\begin{eqnarray*}
D_h(t_7) & = & \left\{ (t_8,t_6) \in \mathbb{R}^2, t_6 >0, h(t_8,t_6,t_7) \leq 1 \right\} \textrm{.}
\end{eqnarray*}
 
\begin{lemma}
\label{D_h lemma}
For $t_7 > 0$ and $Z_6, Z_8 > 0$, we have
\begin{eqnarray}
\label{Z_6}
\meas \{ (t_8,t_6) \in D_h(t_7), t_6 Z_6 \geq 1 \} & \ll & Z_6 \textrm{,} \\
\label{Z_6'}
\meas \{ (t_8,t_6) \in D_h(t_7), t_6 Z_6 < 1 \} & \ll & Z_6^{-1} \textrm{,} \\
\label{Z_8}
\meas \{ (t_8,t_6) \in D_h(t_7), |t_8| Z_8 < 1 \} & \ll & Z_8^{-2/3} \textrm{.}
\end{eqnarray}
\end{lemma}

\begin{proof}
The condition $t_6^2 \left| t_7^2 + t_6 t_8 \right| \leq 1$ proves that the measure of the set where $t_8$ runs over is less or equal to $2 t_6^{-3}$. Since $|t_8| \leq 1$, this measure is also $\ll t_6^{-2}$. Integrating this quantity over $t_6$ using
$t_6 Z_6 \geq 1$ proves \eqref{Z_6}. Since $|t_8| \leq 1$, the bound \eqref{Z_6'} is clear. Finally, since $|t_8|^{-1} \geq 1$ and
$t_6^2 \left| t_7^2 + t_6 t_8 \right| \leq 1$, we can use lemma \ref{elementary} to deduce that the measure of the set where $t_6$ runs over is bounded by $4 |t_8|^{-1/3}$. The bound \eqref{Z_8} immediately follows since $|t_8| Z_8 < 1$.
\end{proof}

The bound \eqref{Z_6} shows that
\begin{eqnarray*}
g_2 \left( \frac{\eta_7}{Y_7}, \kappa; \boldsymbol{\eta}', B \right) & \ll & \frac{Y_6}{\mathcal{L}} \textrm{.}
\end{eqnarray*}
Therefore, if we assume that $Y_6 < \log(B)^A$, we see that the overall contribution of the main term of
$N(\boldsymbol{\eta}',\eta_7,B)$ is
\begin{eqnarray*}
\sum_{\boldsymbol{\eta}', \eta_7} \frac{Y_6^2 Y_8}{\log(B)^{A} \eta_{10}} & \ll &
\sum_{\boldsymbol{\eta}'} \frac{Y_6^2 Y_7 Y_8}{\log(B)^{A} \eta_{10}} \\
& \ll & \sum_{\eta_2,\eta_3,\eta_4,\eta_5,\eta_{10}} \frac{B}{\boldsymbol{\eta}'^{(0,1,1,1,1,1)}} \\
& \ll & B \log(B)^5 \textrm{,}
\end{eqnarray*}
where we have summed over $\eta_7$ using \eqref{height2} and over $\eta_1''$ using $Y_6 < \log(B)^A$. We thereby assume from now on that
\begin{eqnarray}
\label{new Y}
Y_6 & \geq & \log(B)^A \textrm{.}
\end{eqnarray}
The bound \eqref{Z_6'} shows that replacing the condition $t_6 Y_6 \geq \log(B)^A$ by $t_6 > 0$ in the integral defining $g_2$ in the main term of $N(\boldsymbol{\eta}',\eta_7,B)$ yields an error term whose overall contribution is
\begin{eqnarray*}
\sum_{\boldsymbol{\eta}', \eta_7} \frac{\log(B)^A Y_8}{\eta_{10}} & \ll &
\sum_{\boldsymbol{\eta}'} \frac{\log(B)^A Y_7 Y_8}{\eta_{10}} \\
& \ll & \sum_{\eta_2,\eta_3,\eta_4,\eta_5,\eta_{10}} \frac{B}{\boldsymbol{\eta}'^{(0,1,1,1,1,1)}} \\
& \ll & B \log(B)^5 \textrm{,}
\end{eqnarray*}
where we have summed over $\eta_7$ using \eqref{height2} and over $\eta_1''$ using $Y_6 \geq \log(B)^A$. In a similar fashion, the bound \eqref{Z_8} together with the fact that $Y_8 \geq \log(B)^A$ show that removing the condition $|t_8| Y_8 \geq \log(B)^A$ from the integral defining $g_2$ in the main term of $N(\boldsymbol{\eta}',\eta_7,B)$ also creates an error term whose total contribution is $\ll B \log(B)^5 $. As already said, we have therefore replaced $g_2 \left( \eta_7 / Y_7, \kappa; \boldsymbol{\eta}', B \right)$ by $g_3 \left( \eta_7 / Y_7, \kappa \right)$ in the main term of $N(\boldsymbol{\eta}',\eta_7,B)$.

For fixed $\boldsymbol{\eta}' \in \mathcal{V}$ satisfying \eqref{new Y} and the coprimality conditions \eqref{gcd6} and \eqref{gcd7}, let $N'(\boldsymbol{\eta}',B)$ be the sum of the main term of $N(\boldsymbol{\eta}',\eta_7,B)$ over $\eta_7$, $\eta_7$ being subject to the height condition \eqref{height2} and to the coprimality condition \eqref{gcd5}. Recall the definition \eqref{Xi} of $\Xi$. We prove the following lemma.

\begin{lemma}
\label{summation 7}
For any fixed $A \geq 10$, we have the estimate
\begin{eqnarray*}
N'(\boldsymbol{\eta}',B) & = & \zeta(2)^{-1} \Xi \frac{B}{\boldsymbol{\eta}'^{(1,1,1,1,1,1)}}
g_4 \left( \kappa \right) \Theta(\boldsymbol{\eta}') + R'(\boldsymbol{\eta}',B) \textrm{,}
\end{eqnarray*}
where $\Theta(\boldsymbol{\eta}')$ is a certain arithmetic function defined in \eqref{Theta} and where
\begin{eqnarray*}
\sum_{\boldsymbol{\eta}'} R'(\boldsymbol{\eta}',B) & \ll & B \log(B)^5 \textrm{.}
\end{eqnarray*}
\end{lemma}

Recall the definition \eqref{varphi times} of $\varphi^{\times}$. Using lemma \ref{arithmetic preliminary 2'} to sum over $\eta_7$, we see that for any fixed $A \geq 10$ and $0 < \gamma \leq 1$, we have
\begin{eqnarray*}
N'(\boldsymbol{\eta}',B) & = & \zeta(2)^{-1} \Xi \frac{Y_6 Y_7 Y_8}{\eta_{10}}
g_4 \left( \kappa \right) \theta_1(\boldsymbol{\eta}') \varphi^{\ast}(\eta_2 \eta_4 \eta_{10})
\varphi^{\times}( \eta_2 \eta_3 \eta_4 \eta_5 \eta_{10}) \\
& & + O \left( \frac{Y_6 Y_8}{\eta_{10}} Y_7^{\gamma} \sigma_{- \gamma/2}(\eta_2 \eta_4 \eta_{10})
\sup_{t_7 > 0} g_3 \left( t_7, \kappa \right) \right) \textrm{.}
\end{eqnarray*}
Since $g_3$ obviously satisfies the same bound as $g_2$ in lemma \ref{bounds}, we have
\begin{eqnarray*}
\sup_{t_7 > 0} g_3 \left( t_7, \kappa \right) & \ll &  1 \textrm{.}
\end{eqnarray*}
Let us choose for instance $\gamma = 1/2$. As already explained, since $\sigma_{- 1/4}$ has average order $O(1)$, it can be ignored in the estimation of the total contribution of the error term. This contribution is therefore seen to be
\begin{eqnarray*}
\sum_{\boldsymbol{\eta}'} \frac{Y_6 Y_8 Y_7^{1/2}}{\eta_{10}} & \ll &
\sum_{\eta_1'',\eta_3, \eta_4, \eta_5, \eta_{10}} \frac{B}{\boldsymbol{\eta}'^{(1,0,1,1,1,1)}} \\
& \ll & B \log(B)^5 \textrm{,}
\end{eqnarray*}
where we have summed over $\eta_2$ using $Y_7 \geq 1$. Recalling that 
\begin{eqnarray*}
\frac{Y_6 Y_7 Y_8}{\eta_{10}} & = & \frac{B}{\boldsymbol{\eta}'^{(1,1,1,1,1,1)}} \textrm{,}
\end{eqnarray*}
and setting
\begin{eqnarray}
\label{Theta}
\Theta(\boldsymbol{\eta}') & = & \theta_1(\boldsymbol{\eta}') \varphi^{\ast}(\eta_2 \eta_4 \eta_{10})
\varphi^{\times}( \eta_2 \eta_3 \eta_4 \eta_5 \eta_{10}) \textrm{,}
\end{eqnarray}
we see that we have completed the proof of lemma \ref{summation 7}.

Our next goal is to replace $g_4(\kappa)$ by $g_4(0)$ in the main term of $N'(\boldsymbol{\eta}',B)$ in lemma~\ref{summation 7}. It is convenient to set 
\begin{eqnarray*}
D_h & = & \left\{ (t_8,t_6,t_7) \in \mathbb{R}^3, t_6,t_7>0, h(t_8,t_6,t_7) \leq 1 \right\} \textrm{.}
\end{eqnarray*}

\begin{lemma}
We have the bounds
\begin{eqnarray}
\label{t}
\meas \{ (t_8,t_6,t_7) \in D_h, t \leq \left| t_6 t_8 + t_7^2 \right| \} & \ll & t^{-1/2} \textrm{,} \\
\label{t'}
\meas \{ (t_8,t_6,t_7) \in D_h, t > \left| t_6 t_8 + t_7^2 \right| \} & \ll & t^{1/2} \textrm{.}
\end{eqnarray}
\end{lemma}

\begin{proof}
We note that $t \leq \left| t_6 t_8 + t_7^2 \right|$ and $t_6^2 \left| t_6 t_8 + t_7^2 \right| \leq 1$ imply that $t_6^2 t \leq 1$. This proves the first bound since $t_7, |t_8| \leq 1$. In addition, as already said in the proof of lemma \ref{D_h lemma}, the measure of the set where $t_6$ runs over is less or equal to $4 |t_8|^{-1/3}$. Using the condition $t > \left| t_6 t_8 + t_7^2 \right|$, we see that this measure is also less or equal to $4 \min \left( |t_8|^{-1/3}, t |t_8|^{-1} \right) \leq 4 t^{1/2} |t_8|^{-2/3}$. Integrating over $t_7, |t_8| \leq 1$ completes the proof.
\end{proof}

We see that if we assume that $\kappa > 1$ then the bound \eqref{t} allows us to prove that the contribution of the main term of $N'(\boldsymbol{\eta}',B)$ is $\ll B \log(B)^5 $. We therefore assume from now on that
$\kappa \leq 1$, namely
\begin{eqnarray}
\label{new kappa}
\eta_3 \eta_5 Y_7^2& \geq & \eta_{10}^2 \textrm{.}
\end{eqnarray}
The bound \eqref{t'} therefore shows that replacing $g_4(\kappa)$ by $g_4(0)$ in the main term of $N'(\boldsymbol{\eta}',B)$ in lemma \ref{summation 7} creates an error term whose overall contribution is $\ll~B \log(B)^5$. Recalling the equality \eqref{omega}, we see that we have replaced $g_4(\kappa)$ in the main term of $N'(\boldsymbol{\eta}',B)$ in lemma \ref{summation 7} by
\begin{eqnarray*}
\int \int \int_{t_6, t_7 > 0, h(t_8,t_6,t_7) \leq 1} \D t_8 \D t_6 \D t_7 & = & \frac{\omega_{\infty}}{12} \textrm{.}
\end{eqnarray*}

\subsection{Conclusion}

Recall the definition \eqref{V} of $\mathcal{V}$. It remains to sum the main term of $N'(\boldsymbol{\eta}',B)$ over the $\boldsymbol{\eta}' \in \mathcal{V}$ satisfying \eqref{new Y} and \eqref{new kappa} and the coprimality conditions \eqref{gcd6} and \eqref{gcd7}. It is easy to see that replacing
$\left\{ \boldsymbol{\eta}' \in \mathcal{V}, \eqref{new Y}, \eqref{new kappa} \right\}$ by the region 
\begin{eqnarray*}
\mathcal{V}' & = &  \left\{ \boldsymbol{\eta}' \in \mathbb{Z}_{>0}^6, Y_6 \geq 1, Y_7 \geq 1, Y_8 \geq 1, 
\eta_3 \eta_5 Y_7^2 \geq \eta_{10}^2 \right\} \textrm{,}
\end{eqnarray*}
produces an error term whose overall contribution is $\ll B \log(B)^5 \log(\log(B))$. We redefine the arithmetic function $\Theta$ as being equal to zero if the remaining coprimality conditions \eqref{gcd6} and \eqref{gcd7} are not satisfied. Putting together these remarks and the estimate \eqref{estimate N_{U,H}} and fixing for example $A = 10$, we obtain the following lemma.

\begin{lemma}
\label{final lemma}
We have the estimate
\begin{eqnarray*}
N_{U,H}(B) & = & \zeta(2)^{-1} \Xi \frac{\omega_{\infty}}{3} B
\sum_{\boldsymbol{\eta}' \in \mathcal{V}'} \frac{\Theta(\boldsymbol{\eta}')}{\boldsymbol{\eta}'^{(1,1,1,1,1,1)}}
+ O \left( B \log(B)^5 \log(\log(B)) \right) \textrm{.}
\end{eqnarray*}
\end{lemma}

The end of the paper is dedicated to the completion of the proof of theorem \ref{Manin}. Let us introduce the generalized Möbius function  $\boldsymbol{\mu}$ defined for $(n_1, \dots, n_6) \in \mathbb{Z}_{>0}^6$ by
$\boldsymbol{\mu}(n_1, \dots, n_6) = \mu(n_1) \cdots \mu(n_6)$. We set $\mathbf{k} = (k_1,k_2,k_3,k_4,k_5,k_{10})$ and we define for $s \in \mathbb{C}$ such that $\Re(s) > 1$,
\begin{eqnarray*}
F(s) & = & \sum_{\boldsymbol{\eta}' \in \mathbb{Z}_{>0}^6}
\frac{\left|(\Theta \ast \boldsymbol{\mu})(\boldsymbol{\eta}')\right|}{\eta_1''^s \eta_2^s \eta_3^s \eta_4^s \eta_5^s \eta_{10}^s} \\
& = & \prod_p \left( \sum_{\mathbf{k} \in \mathbb{Z}_{\geq 0}^6}
\frac{\left|(\Theta \ast \boldsymbol{\mu}) \left( p^{k_1},p^{k_2},p^{k_3},p^{k_4},p^{k_5}, p^{k_{10}} \right)\right|}
{p^{k_1 s}p^{k_2 s}p^{k_3 s}p^{k_4 s}p^{k_5 s}p^{k_{10} s}} \right) \textrm{.}
\end{eqnarray*}
It is easy to check that if $\mathbf{k} \notin \{0,1\}^6$ then
$(\Theta \ast \boldsymbol{\mu}) \left( p^{k_1},p^{k_2},p^{k_3},p^{k_4},p^{k_5},p^{k_{10}} \right) = 0$ and if exactly one of the $k_i$ is equal to $1$, then
$(\Theta \ast \boldsymbol{\mu}) \left( p^{k_1},p^{k_2},p^{k_3},p^{k_4},p^{k_5}, p^{k_{10}} \right) \ll 1/p$, so the local factors $F_p$ of $F$ satisfy
\begin{eqnarray*}
F_p(s) & = & 1 + O \left( \frac1{p^{ \min \left( \Re(s)+1, 2 \Re(s) \right)}} \right) \textrm{.}
\end{eqnarray*}
This proves that the function $F$ converges in the half-plane $\Re(s) > 1/2$. This clearly implies that $\Theta$ satifies the assumption of \cite[Lemma $8$]{3A1}. The application of this lemma provides
\begin{eqnarray}
\label{sum1}
\ \ \ \ \ \ \ \sum_{\boldsymbol{\eta}' \in \mathcal{V}'} \frac{\Theta(\boldsymbol{\eta}')}{\boldsymbol{\eta}'^{(1,1,1,1,1,1)}} & = & \alpha \left( \sum_{\boldsymbol{\eta}' \in \mathbb{Z}_{>0}^6}
\frac{(\Theta \ast \boldsymbol{\mu})(\boldsymbol{\eta}')}{\boldsymbol{\eta}'^{(1,1,1,1,1,1)}} \right) \log(B)^6 + O \left( \log(B)^5 \right) \textrm{,}
\end{eqnarray}
where $\alpha$ is the volume of the polytope defined in $\mathbb{R}^6$ by $t_1,t_2,t_3,t_4,t_5,t_{10} \geq 0$ and
\begin{eqnarray*}
3t_1 - 2 t_2 - t_3 + 4 t_4 + 2 t_5 - 3 t_{10} & \leq & 1 \textrm{,} \\
t_2 + 2 t_3 + t_4 + 2 t_5 & \leq & 1 \textrm{,} \\
t_1 + 2 t_2 + t_3 + t_{10} & \leq & 1 \textrm{,} \\
2 t_2 + t_3 + 2 t_4  + t_5 + 6 t_{10} & \leq & 2 \textrm{.}
\end{eqnarray*}
The computation of $\alpha$ can be achieved using Franz's additional \textit{Maple} package Convex \cite{Convex} and we find
$\alpha = 1/2880$, that is to say
\begin{eqnarray}
\label{alpha}
\alpha & = & 3 \alpha(\widetilde{V}) \textrm{.}
\end{eqnarray}
Furthermore, since $\Theta(\boldsymbol{\eta}')$ is independent of $\eta_1''$, setting
$\mathbf{k}' = (k_2,k_3,k_4,k_5,k_{10})$, we have
\begin{eqnarray*}
\sum_{\boldsymbol{\eta}' \in \mathbb{Z}_{>0}^6} \frac{(\Theta \ast \boldsymbol{\mu}) (\boldsymbol{\eta}')}{\boldsymbol{\eta}'^{(1,1,1,1,1,1)}} & = & \prod_p \left( \sum_{\mathbf{k} \in \mathbb{Z}_{\geq 0}^6}
\frac{(\Theta \ast \boldsymbol{\mu}) \left( p^{k_1},p^{k_2},p^{k_3},p^{k_4},p^{k_5},p^{k_{10}} \right)}{p^{k_1}p^{k_2}p^{k_3}p^{k_4}p^{k_5}p^{k_{10}}} \right) \\
& = & \prod_p \left( 1 - \frac1{p} \right)^5 \left( \sum_{\mathbf{k}' \in \mathbb{Z}_{\geq 0}^5} 
\frac{\Theta \left( 1,p^{k_2},p^{k_3},p^{k_4},p^{k_5},p^{k_{10}} \right)}{p^{k_2}p^{k_3}p^{k_4}p^{k_5}p^{k_{10}}} \right) \textrm{.}
\end{eqnarray*}
The calculation of these local factors is long but straightforward and we find
\begin{eqnarray*}
\sum_{\mathbf{k}' \in \mathbb{Z}_{\geq 0}^5}
\frac{\Theta \left( 1,p^{k_2},p^{k_3},p^{k_4},p^{k_5},p^{k_{10}} \right)}{p^{k_2}p^{k_3}p^{k_4}p^{k_5}p^{k_{10}}} & = &
\varphi^{+}(p) \varphi^{\times}(p) \left( 1 - \frac1{p} \right) \left( 1 + \frac{7}{p} + \frac1{p^2} \right) \textrm{.}
\end{eqnarray*}
Since
\begin{eqnarray*}
\varphi^{+}(p) & = & \left( 1 - \frac1{p^2} \right)^{-1} \left( 1 - \frac1{p} \right) \textrm{,}
\end{eqnarray*}
we finally get
\begin{eqnarray}
\label{sum2}
\sum_{\boldsymbol{\eta}' \in \mathbb{Z}_{>0}^6} \frac{(\Theta \ast \boldsymbol{\mu}) (\boldsymbol{\eta}')}{\boldsymbol{\eta}'^{(1,1,1,1,1,1)}} & = & \zeta(2) \Xi^{-1} \prod_p \left( 1 - \frac1{p} \right)^7
 \omega_p \textrm{.}
\end{eqnarray}
Putting together the equalities \eqref{sum1}, \eqref{alpha}, \eqref{sum2} and lemma \ref{final lemma} instantly completes the proof of theorem \ref{Manin}.

\bibliographystyle{is-alpha}
\bibliography{biblio}

\end{document}